\documentclass[a4paper,12pt]{amsart}
\usepackage{geometry,amsmath,amssymb,bbm,amsthm,amscd,graphicx}
\usepackage{setspace} 
\usepackage{colortbl}
\usepackage{ytableau}
\usepackage{tikz}
\usepackage{comment}

\newtheorem{Theorem}{Theorem}[section]
\allowdisplaybreaks

\newtheorem{Lemma}[Theorem]{Lemma}

\newtheorem{Corollary}[Theorem]{Corollary}
\theoremstyle{definition}
\newtheorem{Example}[Theorem]{Example}
\newtheorem{Remark}[Theorem]{Remark}

\renewcommand\Re{{\mathrm{Re}}}

\numberwithin{equation}{section}
\begin{document}
\title[Expressions of Schur multiple zeta-functions by zeta-functions of root systems]{Expressions of Schur multiple zeta-functions of anti-hook type by zeta-functions of root systems\large }
\author{Kohji Matsumoto and Maki Nakasuji}
\thanks{The first author is supported by Grant-in-Aid for Scientific Research (B) 18H01111, and the second author is supported by Grant-in-Aid for Scientific Research (C) 18K03223.}

\date{}
\maketitle
%
%
\vskip 1cm
\par\noindent

\begin{abstract} 
We investigate relations among Schur multiple zeta functions and zeta-functions of root systems attached to semisimple Lie algebras.
Schur multiple zeta functions are defined as sums over semi-standard Young tableaux.
Then, assuming the Young tableaux is of anti-hook shape, 
we show that they can be written in terms of modified zeta-functions of root systems of type $A$.
Our proof is quite computational, but we also give a pictorial interpretation of
our argument in terms of Young tableaux.
It is also possible to understand that one of our theorems gives an expression of
Schur multiple zeta functions by an analogue of Weyl group multiple Dirichlet series in the sense of Bump et al.
 By combining with a result of Nakasuji, Phuksuwan and Yamasaki, our theorems yield a new method of finding functional relations among zeta-functions of root systems.
 \end{abstract}

{\small{Keywords: {
Schur multiple zeta functions, zeta-functions of root systems, 
Euler-Zagier multiple zeta-functions, functional relations,
Weyl group multiple Dirichlet series, harmonic product}
 

%
%
%
\section{Introduction}

The theory of various multiple zeta-functions has recently been studied quite
extensively.     An important class of multiple zeta-functions are 
Euler-Zagier multiple zeta-functions, which are a natural generalization of the Riemann zeta-function, and 
relations among their special values are 
of great interest.

The first-named author, with Y. Komori and H. Tsumura, introduced the notion
of zeta-functions of root systems (see \cite{KomoriMatsumotoTsumura} \cite{KMT-JMSJ}), 
which is a further generalization
of Euler-Zagier multiple zeta-functions, and is also regarded as a multi-variable version
of Witten zeta-functions (see \cite{Witten}) attached to semisimple Lie algebras.    
It has been shown that various
relations among multiple zeta values can be regarded as
special cases of functional relations among zeta-functions of root systems, 
whose proof is based on properties of Weyl groups.

On the other hand,
the second-named author, with O. Phuksuwan, and Y. Yamasaki,
introduced another class of representation-theoretic multiple zeta-functions, called
Schur multiple zeta-functions (see \cite{NPY}), associated with semi-standard Young tableaux.    This connects multiple zeta values and its variant multiple
zeta-star values in a natural way.
For recent developments in the theory of Schur multiple zeta-functions, see
\cite{Bach} \cite{BachY} \cite{NN}.

The aim of the present paper is to show relations among these two
representation-theoretic multiple zeta-functions.     Our main theorems
imply that Schur multiple zeta-functions of anti-hook type can be written
in terms of (generalized, or modified) zeta-functions of root systems of type $A$.
These theorems, combined with a result in \cite{NPY}, give a new method of
finding functional relations among zeta-functions of root systems.

First of all we have to define those two classes of multiple zeta-functions.
This will be done in the next section.
Then we will state the first main theorem (Theorem \ref{Thexpression}) and the second main theorem (Theorem \ref{HurwitzTheorem}) in Section
\ref{sectionresults} and in Section \ref{sectionresults2}, respectively, with the
discussion of some consequences of those theorems.
The proofs of the first and the second main theorems will be given in
Section \ref{proof_FMT} and Section \ref{proof_SMT}, respectively.
Our proofs of those theorems are quite computational, but in Section \ref{pictorial}
we will give a pictorial interpretation of
our argument in terms of Young tableaux.

It is to be stressed that in the statement of Theorem \ref{HurwitzTheorem}, an
analogue of Weyl group multiple Dirichlet series in the sense of D. Bump et al.
(see \cite{Bump}) appears.    This point will also be discussed in
Section \ref{sectionresults2}.    

When the length of the associated Young tableaux is small, the structure is easier to
analyze, and therefore it is possible to obtain some simpler expressions.     This topic
will be considered in the last section.

\section{Definitions of relevant multiple zeta-functions}\label{sec2}

Let $\mathbb{N}$, $\mathbb{C}$ be the set of positive integers, and of complex numbers,
respectively.

The most fundamental class of multiple zeta-functions is that of Euler-Zagier multiple
zeta-functions, defined by
\begin{align}\label{EZ_def}
\zeta_{EZ,r}(s_1,\ldots,s_r)=\sum_{1\leq m_1<\cdots<m_r}
{m_1^{-s_1}\cdots
m_r^{-s_r}}
\end{align}
for $r\geq 1$, where $s_1,\ldots,s_r$ are complex variables.    The ``star''-variant
of Euler-Zagier $r$-variable zeta-functions is 
\begin{align}\label{EZstar_def}
\zeta_{EZ,r}^{\star}(s_1,\ldots,s_r)=\sum_{1\leq m_1\leq \cdots\leq m_r}
{m_1^{-s_1}\cdots m_r^{-s_r}}.
\end{align}

We next define Schur multiple zeta functions of anti-hook shaped type.
Let $\lambda=(\lambda_1, \cdots, \lambda_m)$ be a non-increasing sequence of 
$n\in\mathbb{N}$, i.e. $\lambda_1\geq \lambda_2\geq \cdots \lambda_m\geq 1$ with $|\lambda|:=\sum_i\lambda_i=n$.
Then a {\it Young diagram} of shape $ \lambda$ is obtained by drawing $\lambda_i$ boxes in the $i$-th row.
We may identify $\lambda$ with the Young diagram of shape $\lambda$, say $\{(i, j)\in {\mathbb Z}^2 | 1\leq i\leq m, 1\leq j\leq \lambda_i\}$.
A {\it Young tableau} is a filling of the diagram.
If $X$ is some set, 
we call $T= (t_{ij})$, which is obtained by putting $t_{ij}\in X$ into box $(i, j)$ of $\lambda$, a  {\it Young tableau of shape $\lambda$ over $X$}.
Let $T(\lambda,X)$ be the set of all Young tableaux of shape $\lambda$ over $X$, and $\mathrm{SSYT}(\lambda)\subset T(\lambda,\mathbb{N})$ the set of all semi-standard Young tableaux of shape $\lambda$ which satisfies
(i) weakly increasing across each row (ii) strictly increasing down each column.
For  ${\pmb s}=(s_{ij})\in T(\lambda,\mathbb{C}),$ the Schur multiple zeta-function associated with $\lambda$ is defined as in \cite{NPY} by the series 
$$
\zeta_{\lambda}({ \pmb s})=\sum_{M\in \mathrm{SSYT}(\lambda)}
{M^{ -\pmb s}}, 
$$
where $M^{ -\pmb s}=\displaystyle{\prod_{(i, j)\in \lambda}m_{ij}^{-s_{ij}}}$ for $M=(m_{ij})\in \mathrm{SSYT}(\lambda)$. This series converges absolutely if ${\pmb s}\in W_{\lambda}$ where 
\[
  W_\lambda =
\left\{{\pmb s}=(s_{ij})\in T(\lambda,\mathbb{C})\,\left|\,
\begin{array}{l}
 \text{$\Re(s_{ij})\ge 1$ for all $(i,j)\in \lambda \setminus C(\lambda)$ } \\[3pt]
 \text{$\Re(s_{ij})>1$ for all $(i,j)\in C(\lambda)$}
\end{array}
\right.
\right\}.
\]
 with $C(\lambda)$ being the set of all corners of $\lambda$.
 
If the diagram consists of just one column, $\zeta_{\lambda}({ \pmb s})$ is nothing but
the Euler-Zagier multiple zeta-function \eqref{EZ_def}, and if the diagram consists of one row, 
$\zeta_{\lambda}({ \pmb s})$ is the star-variant \eqref{EZstar_def}.
Schur multiple zeta-functions naturally interpolate between these two fundamental notions
\eqref{EZ_def} and \eqref{EZstar_def}.
 
 A skew Young diagram $\theta$ is a diagram obtained as a set difference of two Young diagrams of partitions $\lambda$ and $\mu$
 satisfying $\mu\subset \lambda$, that is $\mu_i\le \lambda_i$ for all $i$.
 In this case, we write $\theta=\lambda/\mu$ and $|\theta|=|\lambda|-|\mu|$.
 Specifically, for $k, \ell \in {\mathbb N}$, if 
 $\lambda=(
\underbrace{k+1, \cdots, k+1}_{\ell+1 \;{\rm times}}
)$
 and 
  $\mu=
(\underbrace{k, \cdots, k}_{\ell\;{\rm times}}
)$
, we call $\lambda/\mu$ an {\it anti-hook type Young diagram} and write $\theta={\rm rib}(k | \ell)$.
The following tableau is of type $\theta={\rm rib}(4 | 3)$.
 $$
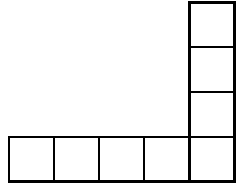
\ytableausetup{boxsize=normal}  
\begin{ytableau}
  \none & \none & \none & \none&   \\
  \none & \none & \none & \none&   \\
  \none & \none & \none & \none &  \\
  & &  &  & 
\end{ytableau}$$
 We similarly use the notation $T(\theta,X)$ for a set $X$, and
 $\mathrm{SSYT}(\theta)$,
 to denote the set of all tableaux over $X$, and the set of all semi-standard Young tableaux of shape $\theta$, respectively.

 Let ${\pmb s}=(s_{ij})\in T(\theta,\mathbb{C})$.
 We define the Schur multiple zeta-function associated with $\theta$ by
\begin{equation}
\label{def:skewSMZ}
 \zeta_{\theta}({\pmb s})
=\sum_{M\in \mathrm{SSYT}(\theta)}
{M^{-\pmb s}}.
\end{equation}


Next we proceed to define zeta-functions of root systems.
Let $\frak g$ be a complex semisimple Lie algebra over ${ \mathbb C}$ and
$\widehat{G}$ be the set of equivalence classes of finite dimensional irreducible representations of ${\frak g}$.
The Witten zeta-function
$$
\zeta_W(s; {\frak g})=\sum_{\varphi\in \widehat{G}}(\dim \varphi)^{-s}.
$$
was introduced by Witten in \cite{Witten}. Weyl gave a formula for the dimension of the irreducible representation, and by using this formula,
the Witten zeta-function is expressed in terms of the corresponding root system.
The zeta-functions of root systems was first introduced by Komori, Matsumoto and Tsumura   (\cite{KomoriMatsumotoTsumura})
as a multi-variable version of this expression to analyze the behavior of $\zeta_W(s; { \frak g})$.
Here we review the definition of zeta-functions of root systems following \cite{KomoriMatsumotoTsumura}.

Let $V$ be an $r$-dimensional real vector space equipped with an inner product $\langle \cdot, \cdot \rangle$
.
The dual space $V^*$ is identified with $V$ via this inner product.
Let $\Delta$ be a finite reduced root system in $V$ and
$\Psi=\{\alpha_1, \cdots, \alpha_r\}$ its fundamental system.
We denote by $\alpha^{\vee}=\frac{2\alpha}{\langle \alpha, \alpha \rangle}$ the coroot associated with a root $\alpha$.    The transformation 
$s_{\alpha} : V\to V$ given by $s_{ \alpha}(x)=x-\frac{2\langle \alpha, x \rangle}{\langle \alpha, \alpha\rangle} \alpha$ is called the reflection attached to $\alpha$. 
Let $\Delta_+$ and $\Delta_-$ be the sets of all positive roots and negative roots, respectively.
Let $\Lambda=\{w_1, \cdots, w_r\}$ be the set of fundamental weights defined by
$\langle \alpha_i^{\vee}, w_j\rangle = \delta_{ij}$ (Kronecker's delta).
Then the Witten zeta-function is expressed as 
$$\zeta_W(s, { \frak g})=K({ \frak g})^s\sum_{m_1=1}^{ \infty}\cdots \sum_{m_r=1}^{ \infty}
\prod_{\alpha\in { \Delta}^+}\langle \alpha^{\vee}, m_1w_1+\cdots + m_rw_r\rangle^{-s},
$$
where $K({ \frak g})=\prod_{\alpha\in \Delta_+}\langle \alpha^{ \vee}, \rho\rangle$ with $\rho=w_1+\cdots+w_r$ being the lowest strongly dominant form.
To extend this expression to its multi-variable version, 
let ${\underbar{\bf s}}=(s_{\alpha})_{\alpha\in \Delta_+}\in { \mathbb C}^{|\Delta_+|}$.
Then the zeta-function of the root system $\Delta$ is defined by
$$
\zeta_r({\underbar{\bf s}}, \Delta):=\sum_{m_1=1}^{\infty}\cdots \sum_{m_r=1}^{\infty}\prod_{\alpha\in \Delta_+}
\langle \alpha^{\vee}, m_1w_1+\cdots +m_rw_r\rangle^{-s_{\alpha}}.
$$
It is well known that simple Lie algebras are classified 
into seven types, $A, B, C, D, E, F$ and $G$.    We denote the corresponding root system
as $\Delta=\Delta(\mathcal{X}_r)$, where $\mathcal{X}$ is one of $A,B,\ldots,G$, and we write the associated
zeta-function by
$\zeta_r({\underbar{\bf s}}, \mathcal{X}_r)$ instead of $\zeta_r({\underbar{\bf s}}, \Delta(\mathcal{X}_r))$, for short.
For example, in the case $\Delta=\Delta(A_r)$ each root is parametrized by
$(i,j)$ ($1\leq i,j\leq r+1$, $i\neq j$), and
\begin{align}\label{A_def}
\zeta_r({\underbar{\bf s}}, A_r)=\sum_{m_1=1}^{\infty}\cdots \sum_{m_r=1}^{\infty}\prod_{1\leq i<j\leq r+1}
(m_i+\cdots +m_{j-1})^{-s{(i,j)}},
\end{align}
where $s(i,j)$ is the variable corresponding to the root parametrized by
$(i,j)$.
For each term $(m_i+\cdots +m_{j-1})^{-s{(i,j)}}$, we call $j-i$ the {\it length} of this
term.
Here (and in what follows) we understand that the components of the vector $\underbar{{\bf s}}=(s{(i,j)})$ 
is arranged according to the length of the corresponding term,
that is,
\begin{multline}\label{ordervector}
\underbar{{\bf s}}=(s{(1,2)},s{(2,3)},\ldots,s{(r,r+1)},s{(1,3)},s{(2,4)},\ldots,s{(r-1,r+1)},\;\ldots,\;
\\
s{(1,r)}, s{(2,r+1)},s{(1,r+1)}).
\end{multline}

Zeta-functions of root systems are not only a generalization of Witten zeta-functions,
but also a generalization of Euler-Zagier multiple zeta-functions.     In fact, 
\eqref{EZ_def} can be regarded as a special case of $\zeta_r(\underline{\bf s}, A_r)$.
(Putting all variables $s{(i,j)}=0$ except $(i,j)=(1,2), (1,3),\ldots,(1,r+1)$, we see that
\eqref{A_def} reduces to \eqref{EZ_def}; see \cite{KMT-MZ}.)
It is also possible to regard \eqref{EZ_def} as a special case of $\zeta_r(\underline{\bf s}, C_r)$
(see \cite{KMT-FA}).

In this paper we will also use the following generalizations of \eqref{A_def}:
For $r>0$ and $0\leq d \leq r$, 
\begin{align}\label{Ahalfstar_def}
\zeta_{r, d}^{ \bullet}(\underline{\bf s},A_r)=\underbrace{
\left(\sum_{m_1=0}^{\infty}\cdots \sum_{m_d=0}^{\infty}\right)'}_{d \text{ times}}
\underbrace{\sum_{m_{d+1}=1}^{\infty}\cdots \sum_{m_r=1}^{\infty}}_{r-d \text{ times}}\prod_{1\leq i<j\leq r+1}
(m_i+\cdots +m_{j-1})^{-s{(i,j)}},
\end{align}
where the prime means that 
the terms $(m_i+\cdots +m_{j-1})^{-s{(i,j)}}$, where $1\leq i<j\leq d+1$ and $m_i=\cdots =m_{j-1}=0$, are omitted.
Obviously $\zeta_{r, 0}^{\bullet}(\underline{\bf s},A_r)=\zeta_r(\underline{\bf s}, A_r)$.
And we also introduce
\begin{align}\label{AH_def}
\zeta_r^H(\underline{\bf s},x, A_r)=\sum_{m_1=1}^{\infty}\cdots \sum_{m_r=1}^{\infty}\prod_{1\leq i<j\leq r+1}
(x+m_i+\cdots +m_{j-1})^{-s{(i,j)}}\qquad (x> 0),
\end{align}
and
\begin{align}\label{Ahalfstarbullet_def}
\zeta_{r, d}^{ \bullet, H}(\underline{\bf s}, x, A_r)=\underbrace{
\left(\sum_{m_1=0}^{\infty}\cdots \sum_{m_d=0}^{\infty}\right)}_{d \text{ times}}
\underbrace{\sum_{m_{d+1}=1}^{\infty}\cdots \sum_{m_r=1}^{\infty}}_{r-d \text{ times}}\prod_{1\leq i<j\leq r+1}
(x+m_i+\cdots +m_{j-1})^{-s{(i,j)}},
\end{align}
so 
$\zeta_{r, 0}^{\bullet, H}(\underline{\bf s}, x, A_r)=\zeta_r^H(\underline{\bf s}, x, A_r)$.
When $r=0$, we set $\zeta_0=\zeta_0^H=\zeta_{0, 0}^{ \bullet}=\zeta_{0, 0}^{\bullet, H}=1$.
We can say $\zeta_{r, d}^{\bullet}$ may be regarded as a ``hybrid" version of zeta-functions and zeta-``star"-functions of root systems 
, and $\zeta_{r}^H$ and $\zeta_{r, d}^{\bullet, H}$ may be regarded as a kind of
zeta-functions of root systems of ``Hurwitz type''.   A more general type of multiple zeta-functions of Hurwitz type was 
introduced and used in \cite[Section 8]{KMT-PLMS}
in the study of multiple $L$-functions of root systems.
\section{The first main theorem}\label{sectionresults}

Consider the case of type $\theta={\rm rib}(k | \ell)$.
Let
\begin{equation}\label{stableau}
{\pmb s}=
\ytableausetup{boxsize=normal}  
\begin{ytableau}
  \none & \none & \none&   s_{k \ell}\\
  \none & \none & \none&   \vdots\\
  \none & \none & \none &   s_{k1}\\
 s_{00} & s_{10}  & \cdots & s_{k 0}
\end{ytableau}\quad .
\end{equation}
It is to be noted that here (and in what follows), the numbering of the double indices
is different from that in the preceding section.
We have
\begin{equation}\label{hookSMZ}
 \zeta_{\theta}({\pmb s}) =
 \sum_{M\in {\rm{SSYT}}( \theta)}
 {m_{00}^{-s_{00}}m_{10}^{-s_{10}}\cdots m_{k 0}^{-s_{k 0}}m_{k 1}^{-s_{k 1}}\cdots m_{k \ell}^{-s_{k \ell}}}
 \end{equation}
for ${\pmb s}\in W_{\lambda}$. 

In \cite[(4.2)]{NPY}, it has been shown that any Schur multiple zeta-function has the expression of linear combination of $\zeta_{EZ}$ or $\zeta^{\star}_{EZ}$: 
\begin{eqnarray}
\zeta_\theta({\pmb s}) &=& \sum_{{\bf t} \preceq {\bf s} } \zeta_{EZ, {\rm length}({\bf t})} ({\bf t}),   \label{zetalambdas} \label{linearcombinationEZ}\\
\zeta_\theta({\pmb s}) &=& \sum_{{\bf t} \preceq {\bf s}' } (-1)^{|\theta|-{\rm length}({\bf t})} \zeta_{EZ, {\rm length}({\bf t})}^{\star} ({\bf t}),  \label{linearcombinationEZstar}
\end{eqnarray} 
where ${\rm length}({\bf t})=m$ for ${\bf t}=(t_1, \cdots, t_m)\in {\mathbb C}^m$.
Recall the explanation of notation \mbox{`` $\preceq$ "} in \cite{NPY} here.
For $|\lambda|=n$,
 let $\mathcal{F}(\lambda)$ be the set of all bijections $f:\lambda\to\{1,2,\ldots,n\}$
 satisfying the following two conditions:
\begin{itemize}
\item[(i)]
 for all $i$, $f((i,j))<f((i,j'))$ if and only if $j<j'$, 
\item[(ii)]
 for all $j$, $f((i,j))<f((i',j))$ if and only if $i<i'$.
\end{itemize} 
 Moreover, for $T=(t_{ij})\in T(\lambda,X)$, 
 put
\[
 V(T)=
\left\{\left.
\left(t_{f^{-1}(1)},t_{f^{-1}(2)},\ldots,t_{f^{-1}(n)}\right)\in X^{n}\,\right|\,
f\in \mathcal{F}(\lambda)
\right\}.
\] 
 Furthermore, when $X$ has an addition $+$,
 we write ${\pmb w} \preceq T$ for ${\pmb w}=(w_1,w_2,\ldots,w_m)\in X^m$
 if there exists $(v_1,v_2,\ldots,v_{n})\in V(T)$ satisfying the following:
 for all $1\le k\le m$, there exist $1\le h_k\le m$ and $l_k\ge 0$ such that  
\begin{itemize}
\item[(i)]
 $w_k=v_{h_k}+v_{h_k+1}+\cdots +v_{h_k+l_k}$,
\item[(ii)]
 there are no $i$ and $i'$ such that $i\ne i'$ and $t_{ij},t_{i'j}\in\{v_{h_k},v_{h_k+1},\ldots ,v_{h_k+l_k}\}$ for some $j$,
\item[(iii)]
 $\bigsqcup^{m}_{k=1}\{h_k,h_k+1,\ldots,h_k+l_k\}=\{1,2,\ldots,n\}$.
\end{itemize}

Since $\zeta_{EZ}$ and $\zeta_{EZ}^{\star}$ are known to be continued 
meromorphically to the whole space, 
\eqref{linearcombinationEZ} and \eqref{linearcombinationEZstar} lead to the following
\begin{Lemma}\label{Lemmeroconti}
$\zeta_\theta({\pmb s})$ can be analytically continued to a meromorphic function in the whole space ${\mathbb C}^{k+\ell+1}$.
\end{Lemma}

In the present paper
we will obtain another expression of this $\zeta_{\theta}({\pmb s})$  in terms of Euler-Zagier mutiple zeta-functions, which is our first main result.

\begin{Theorem}\label{Thexpression}
For a diagram of type $\theta={\rm rib}(k | \ell)$ and 
${\pmb s}=(s_{ij})\in T(\theta,\mathbb{C})$ given by \eqref{stableau}, the identity
\begin{equation}\label{thm1}
 \zeta_{\theta}({\pmb s}) =
\sum_{i=0}^{k}(-1)^{k-i}
\zeta_{EZ, i}^{\star}(s_{00}, s_{10}, \cdots, s_{i-1, 0})
\zeta_{EZ, \ell+k-i+1}(s_{k\ell}, s_{k, \ell-1}, \cdots, s_{k0}, s_{k-1, 0}, \cdots, s_{i0}),
\end{equation}
holds in the whole space $\mathbb{C}^{k+\ell+1}$,
where $\zeta^{\star}_{EZ, i}=1$ for $i=0$.
\end{Theorem}

This theorem is essentially a kind of harmonic product formula.
Combining Theorem \ref{Thexpression} with \eqref{linearcombinationEZ} and
\eqref{linearcombinationEZstar},
we obtain the following functional relations among $\zeta_{EZ}$ and 
$\zeta_{EZ}^{\star}$:
\begin{Corollary}\label{functionalrelation}
The following functional relations hold among $\zeta_{EZ}$ and $\zeta_{EZ}^{\star}$:
\begin{eqnarray*}
\sum_{{\bf t} \preceq {\bf s} } \zeta_{EZ, {\rm length}({\bf t})} ({\bf t}) &=&\sum_{i=0}^{k}(-1)^{k-i}
\zeta_{EZ, i}^{\star}(s_{00}, s_{10}, \cdots, s_{i-1, 0})\cdot \\
&&
\zeta_{EZ, \ell+k-i+1}(s_{k\ell}, s_{k, \ell-1}, \cdots, s_{k0}, s_{k-1, 0}, \cdots, s_{i0}),\\
\sum_{{\bf t} \preceq {\bf s}' } (-1)^{|\theta|-{\rm length}({\bf t})} \zeta_{EZ, {\rm length}({\bf t})}^{\star} ({\bf t})
&=& \sum_{i=0}^{k}(-1)^{k-i}
\zeta_{EZ, i}^{\star}(s_{00}, s_{10}, \cdots, s_{i-1, 0})\cdot \\
&&\zeta_{EZ, \ell+k-i+1}(s_{k\ell}, s_{k, \ell-1}, \cdots, s_{k0}, s_{k-1, 0}, \cdots, s_{i0}).
\end{eqnarray*} 
\end{Corollary}
\begin{Example}
When  $(k, \ell)=(1, 1)$ for
${\pmb s}=
\ytableausetup{boxsize=normal}  
\begin{ytableau}
 \none   & s_{11}\\
 s_{00}  & s_{10}
\end{ytableau}$\quad and
\begin{equation}\label{antihookSMZ}
 \zeta_{\theta}({\pmb s}) =
 \sum
 _{\substack{1\leq m_{00}\leq m_{10}\\
1\leq m_{11}<m_{10}}}
 {m_{00}^{-s_{00}}m_{10}^{-s_{10}}m_{11}^{-s_{11}}},
 \end{equation}
from Theorem \ref{Thexpression} we have
$$\zeta_{\theta}({ \pmb s})=-\zeta_{EZ, 3}(s_{11}, s_{10}, s_{00})+\zeta_{EZ, 1}^{\star}(s_{00})\zeta_{EZ, 2}(s_{11}, s_{10}).
$$ 
Using the harmonic product rule for $\zeta_{EZ}$, say
\begin{eqnarray*}
\zeta_{EZ, 1}^{\star}(s_{00})\zeta_{EZ, 2}(s_{11}, s_{10})&=&\zeta_{EZ, 1}(s_{00})\zeta_{EZ, 2}(s_{11}, s_{10})\\
&=&\zeta_{EZ, 3}(s_{00}, s_{11}, s_{10})+\zeta_{EZ, 3}(s_{11}, s_{00}, s_{10})+\zeta_{EZ, 3}(s_{11}, s_{10}, s_{00})\\
&&+\zeta_{EZ, 2}(s_{00}+s_{11}, s_{10})+\zeta_{EZ, 2}(s_{11}, s_{00}+s_{10}),
\end{eqnarray*}
we obtain
\begin{equation}\label{kl=11}
 \zeta_{\theta}({\pmb s}) =
\zeta_{EZ, 3}(s_{00}, s_{11}, s_{10})+\zeta_{EZ, 3}(s_{11}, s_{00}, s_{10})+\zeta_{EZ, 2}(s_{00}+s_{11}, s_{10})+\zeta_{EZ, 2}(s_{11}, s_{00}+s_{10}),
\end{equation}
which equals to $\sum_{{\bf t} \preceq {\bf s} } \zeta_{EZ} ({\bf t})$. 
This is the simplest case of Corollary \ref{functionalrelation}.
\end{Example}

Note that actually
we can obtain every functional relations among the Euler-Zagier multiple zeta functions from their harmonic product formulas (see Ikeda and Matsuoka \cite{IkedaMatsuoka}).
In this sense, Corollary \ref{functionalrelation} is not an essentially new fact.
However, 
it is still interesting to find a new way of finding such functional relations, as above.
\section{The second main theorem}\label{sectionresults2}
 
Our next theorem is to give an expression of $\zeta_{\theta}({\pmb s})$ in terms of zeta functions of root systems of Hurwitz type, \eqref{AH_def} and \eqref{Ahalfstarbullet_def}. 
 
To state our next theorem, we introduce several notations.  Put 
\[
  W_\lambda^{\circ} =
\left\{{\pmb s}=(s_{ij})\in T(\lambda,\mathbb{C})\,\left|\,
\begin{array}{l}
 \text{$\Re(s_{ij})> 1$ for all $(i,j)\in \lambda$} 
\end{array}
\right.
\right\}.
\]
And let
$
{\bf s}_h={\bf s}_h(x, y) =\underbrace{x, 0, \cdots, 0, y}_h
$. 
If $k<\ell$, we put
\begin{equation}\label{un0}
{\bf u}_n ={\bf u}_n(k, \ell) =\begin{cases}
{\bf s}_{k+\ell+2-n}(0, 0) & (1\leq n \leq k)\\
{\bf s}_{k+\ell+2-n}(s_{k, k+\ell+1-n}, 0) & (k < n\leq \ell)\\
{\bf s}_{k+\ell+2-n}(s_{k, k+\ell+1-n}, s_{n-\ell-1, 0}) & (\ell < n < k+\ell+1)\\
s_{k, 0} & (n = k+\ell+1).
\end{cases}
\end{equation}
If $k\geq \ell$, we put
\begin{equation}\label{un1}
{\bf u}_n ={\bf u}_n(k, \ell)=\begin{cases}
{\bf s}_{k+\ell+2-n}(0, 0) & (1\leq n \leq \ell)\\
{\bf s}_{k+\ell+2-n}(0, s_{n-\ell-1, 0} ) & (\ell < n\leq k)\\
{\bf s}_{k+\ell+2-n}(s_{k, k+\ell+1-n}, s_{n-\ell-1, 0}) & (k < n < k+\ell+1)\\
s_{k, 0} & (n = k+\ell+1).
\end{cases}
\end{equation}
Moreover, if $\nu<\mu$, we put
\begin{equation}\label{vn0}
{\bf v}_n ={\bf v}_n( \nu, \mu; k) =\begin{cases}
0^{\nu}, s_{k0}, 0^{\mu} & (n=0)\\
0^{\nu-n}, s_{kn}, 0^{n-1}, s_{k-n, 0}, 0^{\mu-n} & (1\leq n\leq \nu)\\
0^{\nu}, s_{k-n, 0}, 0^{\mu-n} & (\nu<n\leq \mu) \\
0^{\nu+\mu-n+1} & (\mu<n\leq \mu+\nu+1).
\end{cases}
\end{equation}
If $\nu\geq \mu$, we put
\begin{equation}\label{vn1}
{\bf v}_n ={\bf v}_n( \nu, \mu; k)=\begin{cases}
0^{\nu}, s_{k0}, 0^{\mu} & (n=0)\\
0^{\nu-n}, s_{kn}, 0^{n-1}, s_{k-n, 0}, 0^{\mu-n} & (1\leq n\leq \mu)\\
0^{\nu-n}, s_{kn}, 0^{\mu} & (\mu<n\leq \nu) \\
0^{\nu+\mu-n+1} & (\nu <n\leq \mu+\nu+1).
\end{cases}
\end{equation}
Here, $0^{\nu}$ means
$
\underbrace{0, \cdots, 0}_{ \nu}
$.

\begin{Theorem}\label{HurwitzTheorem}
Let $\theta={\rm rib}(k | \ell)$, 
${\pmb s}=(s_{ij})\in T(\theta,\mathbb{C})$,
\begin{align*}
{\frak s}_1 &={\frak s}_1(\mu; k)=({\bf s}_{k-(\mu+1)}(s_{10}, 0), {\bf s}_{k-(\mu+1)-1}(s_{20}, 0), \cdots, {\bf s}_{2}(s_{k-(\mu+1)-1, 0}, 0), s_{k-(\mu+1), 0}), \\
{\frak s}_2&={\frak s}_2(\nu; k, \ell)=
({\bf s}_{\ell-\nu-1}(s_{k, {\ell-1}},0), {\bf s}_{\ell-\nu-2}(s_{k, {\ell-2}},0)\cdots, 
{\bf s}_{2}(s_{k, {\nu+2}},0),  s_{k,\nu+1}),
\end{align*} 
${\frak u}={\frak u}(k, \ell)=({\bf u}_1, {\bf u}_2, \cdots, {\bf u}_{k+\ell+1})$, 
and 
${\frak v}={\frak v}(\nu, \mu; k)=({\bf v}_0, {\bf v}_1, \cdots, {\bf v}_{\mu+\nu+1})$, where ${\bf u}_n$ {\rm(}and ${\bf v}_n${\rm)} is defined by \eqref{un0} and \eqref{un1} {\rm(}\eqref{vn0} and \eqref{vn1}{\rm)}, respectively.    Then 
for ${\pmb s}\in W_{\theta}^{\circ}$, 
we have
\begin{eqnarray}
 \zeta_{\theta}({\pmb s})  = 
 \zeta_{k+\ell+1, {k+1}}^{\bullet}({\frak u}, A_{k+\ell+1})
+
\sum_{\mu=0}^{k-1}\sum_{\nu=0}^{\ell-1}(-1)^{\mu+\nu}
Z_{\mu,\nu}({\pmb s};(k,\ell)),
\label{thm2}
\end{eqnarray}
where
\begin{align}
Z_{\mu,\nu}({\pmb s};(k,\ell)) = &
\sum_{m_{00}, m_{k\ell}\geq 1}
m_{00}^{-s_{00}}m_{k\ell}^{-s_{k\ell}}\notag\\
& \cdot \zeta^{\bullet, H}_{k-(\mu+1), k-(\mu+1)}({\frak s}_1, m_{00}, A_{k-(\mu+1)})\notag\\
&
\cdot \zeta_{\mu+\nu+1, \nu}^{\bullet, H}
({\frak v}, m_{00}+m_{k\ell}, A_{\mu+\nu+1})
\zeta_{\ell-(\nu+1)}^H({\frak s}_2, m_{k\ell}, A_{\ell-(\nu+1)}).\label{Zmunu}
\end{align}
\label{Z_def}
\end{Theorem}

\begin{Remark}
The first term on the right-hand side of \eqref{thm2} can be written as a sum of
standard (that is, without bullet) zeta-functions of root systems of the form \eqref{A_def} (see Theorem \ref{ThI}).\\
\end{Remark}

\begin{Remark}
In Section \ref{pictorial} we will see that the right-hand side of \eqref{Zmunu}
can be rewritten as a sum involving certain modified Schur-type multiple zeta-functions
of hook type.
\end{Remark}

Theorem \ref{HurwitzTheorem} gives another new expression of 
$\zeta_{\theta}({\pmb s})$.    Therefore, similar to Corollary \ref{functionalrelation},
it produces new functional relations:

\begin{Corollary}
The right-hand sides of \eqref{linearcombinationEZ}, \eqref{linearcombinationEZstar}, \eqref{thm1}  and \eqref{thm2} 
coincide with each other. 
\end{Corollary}

Here we mention the observation that the double series $Z_{\mu,\nu}({\pmb s};(k,\ell))$ 
can be regarded as an
analogue of ``Weyl group multiple Dirichlet series'' in the sense of 
Bump \cite{Bump}.
The origin of the theory of Weyl group multiple Dirichlet series goes back to the work
of Goldfeld and Hoffstein \cite{GoldfeldHoffstein}, in which the two-variable Dirichlet
series
$\sum_{d\geq 1}L(w,d)d^{1/2-2s}$
was studied, where $L(w,d)$ is essentially the Dirichlet $L$-function attached to
the quadratic character $\chi_d$.    Later this study was vastly generalized,
and the multiple series of the form
\begin{align}\label{MDseries}
\sum_{d_i\geq 1 (1\leq i \leq r)}\Phi(w_1,\ldots,w_r;d_1,\ldots,d_r)
d_1^{-s_1}\cdots d_r^{-s_r},
\end{align}
with some zeta-like quantity $\Phi(w_1,\ldots,w_r;d_1,\ldots,d_r)$
parametrized by $d_1,\ldots,d_r$, 
have been studied extensively (for example, see \cite{BBF}) under the name of
Weyl group multiple Dirichlet series.    In their study, representation-theoretic
viewpoints (such as actions of Weyl groups) are quite important.

In \cite[p.19]{Bump}, Bump raised the question of whether there are some undiscovered 
connections between the theory of Weyl group multiple Dirichlet series and the theory
of zeta-functions of root systems.
Now we find that $Z_{\mu,\nu}({\pmb s};(k,\ell))$ is of the form \eqref{MDseries}, with
$r=2$ and $\Phi(w_1,\ldots,w_r;d_1,\ldots,d_r)$ is
the product of three (modified) zeta-functions of root systems.    
Therefore we may say that Theorem \ref{HurwitzTheorem} shows the first link between
these two representation-theoretic multiple series, and thereby 
gives a partial answer to Bump's question.

\begin{Remark}
Lemma \ref{Lemmeroconti} asserts that $\zeta_{\theta}({\pmb s})$
can be continued meromorphically to the whole space ${\mathbb C}^{k+\ell+1}$.
We will see later (Remark \ref{continuation_of_bullet}) that
$\zeta_{k+\ell+1, {k+1}}^{\bullet}({\frak u}, A_{k+\ell+1})$ is also continued
meromorphically to the whole space.
Therefore \eqref{thm2} gives the meromorphic continuation of the sum
$\displaystyle{\sum_{\mu=0}^{k-1}\sum_{\nu=0}^{\ell-1}(-1)^{\mu+\nu}
Z_{\mu,\nu}({\pmb s};(k,\ell))}$ to the whole space ${ \mathbb C}^{k+\ell+1}$.
\end{Remark}

The Schur multiple zeta value has an iterated integral representation when it is of ribbon type (see \cite{KanekoYamamoto}, \cite{Yamamoto} and \cite{NPY}).
 Specifically, the integral expression for the Schur multiple zeta value of anti-hook type is written as follows.    When $\theta={\rm rib}(k | \ell)$ and  
  
\begin{equation}\label{atableau}
{\pmb \alpha}=
\ytableausetup{boxsize=normal}  
\begin{ytableau}
  \none & \none & \none&   \alpha_{k \ell}\\
  \none & \none & \none&   \vdots\\
  \none & \none & \none &   \alpha_{k1}\\
 \alpha_{00} & \alpha_{10}  & \cdots & \alpha_{k 0}
\end{ytableau}\quad ,
\end{equation}
we have
 \begin{eqnarray}
\label{for:intSMZantihook}
\zeta_{\theta}(\pmb \alpha)
&&=
\int_{\Delta({\pmb y})}
\prod_{\iota =0}^{\ell}
\left(
\frac{dy_{\alpha_{k, \ell+1}+\cdots + \alpha_{k, \ell-\iota+1}+1}}{1-y_{\alpha_{k, \ell+1}+\cdots + \alpha_{k, \ell-\iota+1}+1}}
\prod^{\alpha_{k, \ell+1}+\cdots + \alpha_{k, \ell-\iota}}_{j=\alpha_{k, \ell+1}+\cdots + \alpha_{k, \ell-\iota+1}+2}\frac{dy_j}{y_j}
\right)
\notag\\
&&\ \ \ \cdot
\prod^{k-1}_{\kappa=0}
\left(
\frac{dy_{\alpha_{k, \ell}+\cdots + \alpha_{k1} + \alpha_{k0} +\cdots +\alpha_{k-\kappa, 0}+1}}{1-y_{\alpha_{k, \ell}+\cdots + \alpha_{k1} + \alpha_{k0} +\cdots +\alpha_{k-\kappa, 0}+1}}
\prod^{\alpha_{k, \ell}+\cdots + \alpha_{k 1} + \alpha_{k 0}+\cdots + \alpha_{k-\kappa-1, 0}}_{j=\alpha_{k, \ell}+\cdots + \alpha_{k1} + \alpha_{k0}+\cdots +\alpha_{k-\kappa, 0}+2}\frac{dy_j}{y_j}
\right),\notag\\
\end{eqnarray}
 where we set $\alpha_{k, \ell+1}=0$ and 
 $$
 \Delta({\pmb y})
=\left\{\begin{array}{ll}
&(y_1,\ldots, y_{|{\pmb \alpha|}})\in[0,1]^{|{\pmb \alpha|}}\, \quad {\rm s.t.}\, 
\\
& y_1<y_2<\cdots <y_{\alpha_{k\ell}+\cdots \alpha_{k0}}> y_{\alpha_{k\ell}+\cdots \alpha_{k0}+1}, \\
& y_{\alpha_{k\ell}+\cdots \alpha_{k0}+1}<y_{\alpha_{k\ell}+\cdots \alpha_{k0}+2} < \cdots < y_{\alpha_{k\ell}+\cdots \alpha_{k0}+\alpha_{k-1,0}}
> y_{\alpha_{k\ell}+\cdots \alpha_{k0}+\alpha_{k-1,0}+1}, \\
& \cdots \\
& y_{\alpha_{k\ell}+\cdots +\alpha_{k0}+\alpha_{k-1, 0}+\cdots \alpha_{10}+1}<y_{\alpha_{k\ell}+\cdots +\alpha_{k0}+\alpha_{k-1, 0}+\cdots \alpha_{10}+2} < \cdots < y_{\alpha_{k\ell}+\cdots +\alpha_{k0}+\alpha_{k-1, 0}+\cdots \alpha_{00}}
\end{array}\right\},
$$
with $|{\pmb \alpha|}=\alpha_{00}+\alpha_{10}+\cdots + \alpha_{k0}+\alpha_{k1}+\cdots + \alpha_{k\ell}$.
Using the change of variables $y_i'=1-y_{|{\pmb \alpha}|+1-i}$ for $1\leq i \leq |{\pmb \alpha}|$, we have a {\it duality} (see also \cite{NPY}). 
This kind of duality also leads us to new relations.

\begin{Example}\label{Ex4-1}
When $(k, \ell)=(1, 1)$, Theorem \ref{HurwitzTheorem} implies 
$$\zeta_{\theta}({\pmb s})=\zeta^{\bullet}_{3,2}(0^3 | s_{11}, s_{00} | s_{10}, A_3)
+Z_{00}({\pmb s};(1,1)),
$$
and
\begin{align*}
Z_{00}({\pmb s};(1,1))&=\sum_{m_{00}, m_{11}\geq 1}m_{00}^{-s_{00}}m_{11}^{-s_{11}}\zeta^{\bullet, H}_{0,0}({\frak s}_1, m_{00}, A_0)\zeta^{\bullet, H}_{1, 0}({ \frak v}, m_{00}+m_{11}, A_1)\zeta^{H}_{0}({\frak s}_2, m_{11}, A_0)\\
&=
\sum_{m, m_{00}, m_{11}\geq 1}m_{00}^{-s_{00}}m_{11}^{-s_{11}}(m_{00}+m_{11}+m)^{-s_{10}}=\zeta_3(s_{00}, s_{11}, 0 | 0^2 | s_{10}, A_3).
\end{align*} 
In the above formulas, we replace the commas between ${\bf s}_i(\cdot, \cdot)$ and ${\bf s}_{i-1}(\cdot, \cdot)$ by (vertical) bars to guide the eye.
Therefore
\begin{align}\label{example4-1}
\zeta_{\theta}({\pmb s})&=\zeta^{\bullet}_{3,2}(0^3 | s_{11}, s_{00} | s_{10}, A_3)
+\zeta_3(s_{00}, s_{11}, 0 | 0^2 | s_{10}, A_3)\\
&=\zeta_3(0^3 | s_{11}, s_{00} | s_{10}, A_3)+\zeta_2(s_{11}, 0 | s_{00}+s_{10}, A_2)
+\zeta_2(s_{11}, s_{00} | s_{10}, A_2)\notag\\
&\qquad +\zeta_3(s_{00}, s_{11}, 0 | 0^2 | s_{10}, A_3),\notag
\end{align}
where the second equality is by Theorem \ref{ThII1} and Theorem \ref{ThI}.
Combining with \eqref{kl=11}, we obtain
\begin{multline}\label{korekore}
\zeta_{EZ, 3}(s_{00}, s_{11}, s_{10})+\zeta_{EZ, 3}(s_{11}, s_{00}, s_{10})+\zeta_{EZ, 2}(s_{00}+s_{11}, s_{10})+\zeta_{EZ, 2}(s_{11}, s_{00}+s_{10})\\
=\zeta_3(0^3 | s_{11}, s_{00} | s_{10}, A_3)+\zeta_2(s_{11}, 0 | s_{00}+s_{10}, A_2)
+\zeta_2(s_{11}, s_{00} | s_{10}, A_2)
+\zeta_3(s_{00}, s_{11}, 0 | 0^2 | s_{10}, A_3).
\end{multline}
Writing each term on the left-hand side of \eqref{korekore} in terms of zeta-functions of root systems, the left-hand side is 
\begin{equation*}
\begin{array}{ll}
=&
\zeta_3(s_{11}, 0, 0 | s_{00}, 0 | s_{10}, A_3)
+\zeta_2(s_{00}+s_{11}, 0 | s_{10}, A_2)\vspace{3mm}\\
&+\zeta_2(0, s_{11} | s_{00}+s_{10}, A_2)
+\zeta_3(s_{00}, 0, 0 | s_{11}, 0 | s_{10}, A_3). 
\end{array}
\end{equation*}
Therefore, noting 
$\zeta_2(0, s_{11} | s_{00}+s_{10}, A_2)=\zeta_2(s_{11}, 0 | s_{00}+s_{10}, A_2)$, we can rewrite \eqref{korekore} as
$$
\begin{array}{ll}
&\zeta_3(s_{11}, 0, 0 | s_{00}, 0 | s_{10}, A_3)
+\zeta_2(s_{00}+s_{11}, 0 | s_{10}, A_2)
+\zeta_3(s_{00}, 0, 0 | s_{11}, 0 | s_{10}, A_3)\vspace{3mm}\\
& = \zeta_3(0^3 | s_{11}, s_{00} | s_{10}, A_3)
+\zeta_2(s_{11}, s_{00} | s_{10}, A_2)
+\zeta_3(s_{00}, s_{11}, 0 | 0^2 | s_{10}, A_3)
\end{array}
$$
which is a functional relation
among zeta-functions of root systems.

Iterated integral expression shows the duality for Schur multiple zeta values as mentioned in Section \ref{sectionresults}.   For example, applying this formula for
${\pmb s}=
\ytableausetup{boxsize=normal}  
\begin{ytableau}
  \none &   1\\
 2 & 2
\end{ytableau}$, we have
$$
\zeta_{\theta}
\left(\ytableausetup{boxsize=normal}  
\begin{ytableau}
  \none &  1\\
 2 & 2
\end{ytableau}\right)
=\zeta_{(2)}
\left(\ytableausetup{boxsize=normal}  
\begin{ytableau}
 3 &  2
\end{ytableau}\right).
$$
Since$$
\zeta_{(2)}
\left(\ytableausetup{boxsize=normal}  
\begin{ytableau}
 3 &  2
\end{ytableau}\right)
=\zeta^{ \star}_{EZ, 2}(3, 2)=\zeta_{EZ, 2}(3, 2)+\zeta(5).
$$
These relations with \eqref{example4-1} lead to
\begin{align*}
&\zeta_3(0^3 | 1, 2 | 2, A_3)+\zeta_2(1, 0 | 4, A_2)
+\zeta_2(1, 2 | 2, A_2)
+\zeta_3(2, 1, 0 | 0^2 | 2, A_3)\\
&=\zeta_2(3, 0 | 2, A_2)+\zeta_1(5, A_1).
\end{align*}

\end{Example}

\section{Proof of Theorem \ref{Thexpression}}\label{proof_FMT}

Now we start the proofs.
In this and the next section we describe our proofs of main theorems, which are
quite computational.    However it is possible to give a pictorial interpretation
of our argument in terms of Young tableaux.
We will discuss such an interpretation in Section \ref{pictorial}.

The present section is devoted to the proof of Theorm \ref{Thexpression}.
From the definition of semi-standard Young tableaux, $1\leq m_{00}\leq m_{10}\leq \cdots \leq m_{k0}$ and
$1\leq m_{k\ell}< m_{k, \ell-1}< \cdots < m_{k0}$. 
So, setting $m_{10}=m_{00}+a_{1}$ ($a_{1}\geq 0$), 
$m_{20}=m_{00}+a_{1}+a_{2}$ ($a_{1}, a_{2}\geq 0$),
$\cdots$,
\begin{equation}\label{k0ina}
m_{k0}=m_{00}+a_{1}+a_{2}+\cdots +a_{k} \quad (a_{i}\geq 0),
\end{equation}
and  $m_{k, \ell-1}=m_{k \ell}+b_{\ell-1}$ ($b_{\ell-1}\geq 1$), 
$m_{k, \ell-2}=m_{k \ell}+b_{\ell-1}+b_{\ell-2}$ ($b_{\ell-1}, b_{\ell-2}\geq 1$),
$\cdots$,
$m_{k 1}=m_{k \ell}+b_{\ell-1}+b_{\ell-2}+\cdots +b_{1}$ ($b_{j}\geq 1$), 
\begin{equation}\label{k0inb}
m_{k 0}=m_{k \ell}+b_{\ell-1}+b_{\ell-2}+\cdots +b_{1}+b_0\quad (b_{j}\geq 1),
\end{equation}
then we can write \eqref{hookSMZ} as 
\begin{eqnarray}
&&\zeta_{\theta}({\pmb s})
= \sum_{\substack{m_{00}, m_{k\ell} \geq 1\\
a_i\geq 0 (1\leq i \leq k-1)\\
b_j\geq 1(0\leq j\leq \ell-1)\\
1\leq m_{k-1, 0}\leq m_{k0}}}m_{00}^{-s_{00}}(m_{00}+a_1)^{-s_{10}}\cdots (m_{00}+a_1+\cdots + a_{k-1})^{-s_{k-1, 0}}\notag\\
&&\quad\cdot (m_{k\ell}+b_{\ell-1}\cdots + b_0)^{-s_{k0}}
(m_{k\ell}+b_{\ell-1}\cdots + b_1)^{-s_{k1}}\cdots 
 (m_{k \ell}+b_{\ell-1})^{-s_{k, \ell-1}}m_{k \ell}^{-s_{k \ell}}\label{zetaaball}
\end{eqnarray}
for ${\pmb s}\in W_{\theta}$. 
The summation is over $m_{00}, m_{k\ell}, a_i \;(1\leq i \leq k-1)$ and $b_j\; (0\leq j \leq \ell-1)$ which satisfies 
$m_{k-1, 0}=m_{00}+a_1+\cdots + a_{k-1}\leq m_{k\ell}+b_{\ell-1}\cdots + b_0=m_{k0}$.
Temporarily, we assume ${\pmb s}\in W_{\theta}^{\circ}$.
We decompose this sum according to
\begin{align}\label{basicdecomp}
\sum_{1\leq m_{k-1, 0}\leq m_{k0}}=\sum_{m_{k-1, 0}, m_{k0}\geq 1}-\sum_{m_{k-1, 0}>m_{k0}\geq 1},
\end{align}
where in the first sum on the right-hand side no inequality between $m_{k-1,0}$ and
$m_{k0}$ is required. 
Note that it is valid since ${\pmb s}\in W_{\theta}^{\circ}$.
The condition in the second sum $m_{k-1, 0}>m_{k0}$ can be rewritten as 
$m_{k-1, 0}=m_{k0}+p_1$ for $p_1\geq 1$, so \eqref{zetaaball} is,
\begin{eqnarray}
\lefteqn{ \zeta_{\theta}({\pmb s}) }\notag\\
&& =\sum_
{\substack{m_{00}, m_{k\ell} \geq 1\\
a_i\geq 0 (1\leq i \leq k-1)\\
b_j\geq 1(0\leq j\leq \ell-1)\\
m_{k-1, 0}, m_{k0}\geq 1}}m_{00}^{-s_{00}}(m_{00}+a_1)^{-s_{10}}\cdots (m_{00}+a_1+\cdots + a_{k-1})^{-s_{k-1, 0}}
 (m_{k\ell}+b_{\ell-1}\cdots + b_0)^{-s_{k0}}\notag\\
&& \cdot (m_{k\ell}+b_{\ell-1}\cdots + b_1)^{-s_{k1}}\cdots 
 (m_{k \ell}+b_{\ell-1})^{-s_{k, \ell-1}}m_{k \ell}^{-s_{k \ell}}\notag\\
&&
- \sum_
{\substack{m_{00}, m_{k\ell} \geq 1\\
a_i\geq 0 (1\leq i \leq k-2)\\
p_1, b_j\geq 1(0\leq j\leq \ell-1)\\
1\leq m_{k-2, 0}\leq m_{k-1, 0}}}m_{00}^{-s_{00}}(m_{00}+a_1)^{-s_{10}}\cdots 
(m_{00}+a_1+\cdots+a_{k-2})^{-s_{k-2,0}}\notag\\
&&\cdot(m_{k\ell}+b_{\ell-1}\cdots + b_0+p_1)^{-s_{k-1, 0}}
 (m_{k\ell}+b_{\ell-1}\cdots + b_0)^{-s_{k0}} (m_{k\ell}+b_{\ell-1}\cdots + b_1)^{-s_{k1}}\cdots
 \notag\\
&&\cdot (m_{k \ell}+b_{\ell-1})^{-s_{k, \ell-1}}m_{k \ell}^{-s_{k \ell}} .\label{pic5}
\end{eqnarray}
We apply the same argument repeatedly for the summation $\sum_{m_{k-(i+1), 0}\leq m_{k-i, 0}}$ ($1\leq i \leq k-1$), then we obtain
\begin{eqnarray}
 \zeta_{\theta}({\pmb s}) &=&
 \sum_{i=0}^{k}(-1)^{k-i}
 \sum_{\substack{m_{00}, m_{k\ell}\geq 1\\
 a_t\geq 0 (1\leq t \leq i-1)\\
 p_u \geq 1 (1\leq u \leq k-i)\\
 b_j\geq 1 (0\leq j \leq \ell-1)}}m_{00}^{-s_{00}}(m_{00}+a_1)^{-s_{10}}\cdots 
 (m_{00}+a_1+\cdots + a_{i-1})^{-s_{i-1, 0}}\notag\\
 &&
 \cdot(m_{k\ell}+b_{\ell-1}\cdots + b_0+p_1+p_2+\cdots + p_{k-i})^{-s_{i 0}}\notag\\
&& \cdot(m_{k\ell}+b_{\ell-1}\cdots + b_0+p_1+p_2+\cdots + p_{k-(i+1)})^{-s_{i+1, 0}}
\cdots\notag\\
&& \cdot(m_{k\ell}+b_{\ell-1}\cdots + b_0+p_1)^{-s_{k-1, 0}}
 (m_{k\ell}+b_{\ell-1}\cdots + b_0)^{-s_{k0}}\notag\\
&& \cdot (m_{k\ell}+b_{\ell-1}\cdots + b_1)^{-s_{k1}}\cdots 
 (m_{k \ell}+b_{\ell-1})^{-s_{k, \ell-1}}m_{k \ell}^{-s_{k \ell}}.\label{intermsofEZ}
\end{eqnarray} 
On the inner sum, each term can be divided into two factors: the factor involving only
$m_{00}, a_1, \cdots, a_{i-1}$, and the remaining factor involvong only
$m_{k\ell}, b_0, \cdots, b_{\ell-1}, p_1, \cdots, p_{k-i}$.
The summation of the first factor gives the multiple zeta-star function 
$\zeta_{EZ, i}^{\star}(s_{00}, s_{10}, \cdots, s_{i-1, 0})$, and 
the summation over the second factor gives the multiple zeta function 
$\zeta_{EZ, \ell+k-i+1}(s_{k\ell}, s_{k, \ell-1}, \cdots, s_{k0}, s_{k-1, 0}, \cdots, s_{i0})$. Therefore we obtain \eqref{thm1} for ${\pmb s}\in W_{\theta}^{\circ}$.
Lemma \ref{Lemmeroconti} then leads to Theorem \ref{Thexpression}.

\section{Proof of Theorem \ref{HurwitzTheorem}}\label{proof_SMT}
This section is devoted to the proof of Theorem \ref{HurwitzTheorem}. We continue to use the notations of \eqref{k0ina} and \eqref{k0inb} in the previous section.
Now the summation on the right-hand side of \eqref{hookSMZ} can be divided into two cases; (i) $m_{k 0}\leq m_{00}+m_{k \ell}$ and (ii) $m_{k 0}>m_{00}+m_{k \ell}$. Let us write
\begin{equation}\label{i+ii}
 \zeta_{\theta}({\pmb s}) =
{\sum}_{\rm (i)} +{\sum}_{\rm (ii)}\quad\quad({\pmb s}\in W_{\theta}), \end{equation}
where the first and the second sums correspond to the conditions (i) and (ii),
respectively.

In the following two subsections, we will evaluate $\sum_{\rm (i)}$ and $\sum_{\rm (ii)}$, 
respectively (Theorems \ref{ThII1} and \ref{ThII2}).   From these two theorems,
our second main theorem (Theorem \ref{HurwitzTheorem}) follows immediately.
%
\subsection{Evaluation of ${\sum}_{\rm (i)}$}
Consider the case $m_{k 0}\leq m_{00}+m_{k \ell}$.
From the condition, we can say $m_{00}+m_{k \ell}=m_{k 0}+q_1$ ($q_1\geq 0$).
By \eqref{k0ina}, 
we have
$m_{00}+m_{k \ell}=m_{00}+a_{1}+a_{2}+\cdots + a_k+q_1$, which leads to
$m_{k \ell}=q_1+a_{1}+a_{2}+\cdots + a_k$. Note that the case 
\begin{equation}\label{allarezero}
q_1=a_1=\cdots =a_k=0 
\end{equation}
does not occur.
On the other hand, 
by \eqref{k0inb},
we have
$m_{00}+m_{k \ell}=m_{k \ell}+b_{\ell-1}+b_{\ell-2}+\cdots +b_{0}+q_1$, which leads to
$m_{00}=b_{0}+\cdots +b_{\ell-2}+b_{\ell-1}+q_1$.
Then for ${\pmb s}\in W_{\theta}$ we obtain
$$
\begin{array}{ll}
{\displaystyle{\sum}}_{\rm (i)}&=
 \displaystyle{
\sum_{M\in {\rm{SSYT}}( \theta) {\rm and} {\rm (i)}}
{m_{00}^{-s_{00}}m_{10}^{-s_{10}}\cdots m_{k 0}^{-s_{k 0}}m_{k 1}^{-s_{k1}}\cdots m_{k \ell}^{-s_{k \ell}}}}\vspace{3mm}\\
&=
\displaystyle{
\sum_{\substack{a_i\geq 0 (1\leq i \leq k)\\ 
b_j\geq 1 (0\leq j\leq \ell-1)\\
 q_1\geq 0}}}^{\!\!\!\!\!\!\!\!\!\!\!\!\!\!\prime}
\;\;(b_{0}+\cdots +b_{\ell-2}+b_{\ell-1}+q_1)^{-s_{00}}
(b_{0}+\cdots +b_{\ell-2}+b_{\ell-1}+q_1+a_{1})^{-s_{10}}\\
& \hspace{2.8cm}
\cdot(b_{0}+\cdots +b_{\ell-2}+b_{\ell-1}+q_1+a_{1}+a_{2})^{-s_{20}}\cdots\\
& \hspace{2.8cm}
\cdot(b_{0}+\cdots +b_{\ell-2}+b_{\ell-1}+q_1+a_{1}+a_{2}+\cdots+a_k)^{-s_{k0}}\\
& \hspace{2.8cm}
\cdot(b_{1}+\cdots +b_{\ell-2}+b_{\ell-1}+q_1+a_{1}+a_{2}+\cdots+a_k)^{-s_{k 1}}\\
& \hspace{2.8cm}
\cdot(b_{2}+\cdots +b_{\ell-2}+b_{\ell-1}+q_1+a_{1}+a_{2}+\cdots+a_k)^{-s_{k 2}}\\
& \hspace{2.8cm}\cdots \\
& \hspace{2.8cm}
\cdot(b_{\ell-1}+q_1+a_{1}+a_{2}+\cdots+a_k)^{-s_{k, \ell-1}}
(q_1+a_{1}+a_{2}+\cdots+a_k)^{-s_{k \ell}},
\end{array}$$
where the prime means that the sum omits the case of \eqref{allarezero}.
 Changing the order of terms, then for ${\pmb s\in W_{\theta}}$
 \begin{equation}\label{i-part}
\begin{array}{ll}
\displaystyle{{\sum}_{\rm (i)}}&=
 \displaystyle{
\sum_{M\in {\rm{SSYT}}( \theta) {\rm and} {\rm (i)}}
{m_{00}^{-s_{00}}m_{10}^{-s_{10}}\cdots m_{k 0}^{-s_{k 0}}m_{k 1}^{-s_{k1}}\cdots m_{k \ell}^{-s_{k \ell}}}}\vspace{3mm}\\
&=
\displaystyle{
\sum_{\substack{a_i\geq 0 (1\leq i\leq k)\\ b_j\geq 1 (0\leq j\leq \ell-1)\\
q_1\geq 0}}}^{\!\!\!\!\!\!\!\!\!\!\!\!\!\!\prime}
\;\;(a_k+\cdots+a_{2}+a_{1}+q_1)^{-s_{k \ell}}
(a_k+\cdots+a_{2}+a_{1}+q_1+b_{\ell-1})^{-s_{k, \ell-1}}
\\
& \hspace{2.8cm}
\cdot(a_k+\cdots+a_{2}+a_{1}+q_1+b_{\ell-1}+b_{\ell-2})^{-s_{k, \ell-2}}\\
& \hspace{2.8cm}\cdots \\
& \hspace{2.8cm}
\cdot(a_k+\cdots+a_{2}+a_{1}+q_1+b_{\ell-1}+b_{\ell-2}+\cdots +b_{1})^{-s_{k 1}}\\
& \hspace{2.8cm}
\cdot(a_k+\cdots+a_{2}+a_{1}+q_1+b_{\ell-1}+b_{\ell-2}+\cdots +b_{0})^{-s_{k0}}\\
& \hspace{2.8cm}
\cdot(a_{k-1}+\cdots+a_{2}+a_{1}+q_1+b_{\ell-1}+b_{\ell-2}+\cdots +b_{0})^{-s_{k-1,0}}\\
& \hspace{2.8cm}\cdots \\
& \hspace{2.8cm}
\cdot(a_{2}+a_{1}+q_1+b_{\ell-1}+b_{\ell-2}+\cdots +b_{0})^{-s_{20}}\\
& \hspace{2.8cm}
\cdot
(a_{1}+q_1+b_{\ell-1}+b_{\ell-2}+\cdots+b_{0} )^{-s_{10}}\\
& \hspace{2.8cm}
\cdot(q_1+b_{\ell-1}+b_{\ell-2}+\cdots +b_{0})^{-s_{00}}.
\end{array}
\end{equation}
This right-hand side can be written by using the notation given by \eqref{Ahalfstar_def}
with $r=k+\ell+1$, $d=k+1$.    The correspondence of the parameters is
\begin{align*}
(m_1,\ldots,m_k,m_{k+1}(=m_d),m_{k+2},\ldots,m_{k+\ell+1})
\leftrightarrow (a_k,\ldots,a_1,q_1,b_{\ell-1},\ldots,b_0).
\end{align*}

\noindent Then using the notation of \eqref{Ahalfstar_def}, \eqref{un0} and \eqref{un1}, we obtain
\begin{Theorem}\label{ThII1}
For ${\pmb s}\in W_{\theta}$,  
\begin{equation}\label{combimultizeta}
{\sum}_{\rm (i)}=
\zeta_{k+\ell+1, {k+1}}^{\bullet}({\bf u}_1, {\bf u}_2, \cdots, {\bf u}_{k+\ell+1}, A_{k+\ell+1}).
\end{equation}
\end{Theorem}

\begin{proof}
Consider the case $k<\ell$.
On the right-hand side of \eqref{i-part}, there is no term of length $\leq k$.
Therefore all components of ${\bf u}_n$ are zero for $1\leq n\leq k$.     When $k<n\leq \ell$, 
just one term of length $n$ appears, which is
$$(a_k+\cdots+a_1+q_1+b_{\ell-1}+\cdots+b_{k+\ell+1-n})^{-s_{k,k+\ell+1-n}},$$
and this corresponds to the fact that, for $k<n\leq \ell$, ${\bf u}_n$ has just one non-zero
component $s_{k,k+\ell+1-n}$.    In view of the rule \eqref{ordervector}, this component
is the left-most component of ${\bf u}_n$.
When $\ell<n< k+\ell+1$, there are two terms of length $n$, hence ${\bf u}_n$ has two non-zero
components.
Lastly ${\bf u}_n$ for $n=k+\ell+1$ corresponds to the term
$$(a_k+\cdots+a_{2}+a_{1}+q_1+b_{\ell-1}+b_{\ell-2}+\cdots +b_{0})^{-s_{k0}}.$$
The case $k\geq\ell$ is similar.
\end{proof}

It is to be noted that
the equation \eqref{combimultizeta} can also be written in terms of only zeta-functions of root systems $\zeta_r(\underline{\bf s}, A_r)$. 
Put $a_0=q_1$. Define
$I_0^{[0,k]}=\phi$, and
$$I_m^{[0,k]}=\{ \{i_1, \cdots, i_m\}\in {\mathbb N}^m | 
0\leq i_1<\cdots <i_m\leq k\}\qquad(1\leq m\leq k).$$
For $I\in I_m^{[0,k]}$, put $J_I=\{0, \cdots, k\}\backslash I$.
We divide the right-hand side of \eqref{i-part} as
$$
\sum_{m=0}^{k}\sum_{I\in I_m^{[0, k]}}S_I,
$$
where $S_I$ is the sum of all terms satisfying $a_i=0$ for all $i\in I$ and 
$a_j\geq 1$ for all $j\in J_I$.

We will show that each $S_I$ is a zeta-function of a root system.
For each $I$, write 
$$J_I=\{\alpha_0, \alpha_1, \cdots, \alpha_{k-m}\} \qquad
(\alpha_0<\alpha_1<\cdots < \alpha_{k-m}),$$
and let\\
$$
z_{n}^I =z_{n}^I(k, m)=
\left\{
\begin{array}{ll}
s_{00}+s_{10}+\cdots+s_{\alpha_0-1,0} & (n=0),\\
s_{\alpha_{n-1},0}+s_{\alpha_{n-1}+1, 0}+\cdots + s_{\alpha_{n}-1, 0} &  
(1\leq n \leq k-m),\\
s_{\alpha_{n-1},0}+s_{\alpha_{n-1}+1, 0}+\cdots + s_{k, 0} & (n=k-m+1).
\end{array}\right.
$$
Here, if $\alpha_0=0$ then we understand $z_0^I=0$.

If $k-m<\ell$, we put
\begin{equation}\label{un2}
{\bf u}_n^I={\bf u}_n^I(k, \ell, m)=\begin{cases}
{\bf s}_{k-m+\ell+2-n}(0, 0) & (1\leq n \leq k-m)\\
{\bf s}_{k-m+\ell+2-n}(s_{k, \;k-m+\ell+1-n}, 0) & (k-m < n<  \ell)\\
{\bf s}_{k-m+\ell+2-n}(s_{k, \;k-m+\ell+1-n}, z_{n-\ell}^I) & (\ell \leq n < k-m+\ell+1)\\
z_{n-\ell}^I & (n = k-m+\ell+1).
\end{cases}
\end{equation}
If $k-m\geq \ell$, we put
\begin{equation}\label{un3}
{\bf u}_n^I={\bf u}_n^I(k, \ell, m)=\begin{cases}
{\bf s}_{k-m+\ell+2-n}(0, 0) & (1\leq n < \ell)\\
{\bf s}_{k-m+\ell+2-n}(0, z_{n-\ell}^I) & (\ell \leq n\leq k-m)\\
{\bf s}_{k-m+\ell+2-n}(s_{k, \;k-m+\ell+1-n}, z_{n-\ell}^I) & (k-m < n < k-m+\ell+1)\\
z_{n-\ell}^I & (n = k-m+\ell+1).
\end{cases}
\end{equation}
Note that ${\bf u}_n^I={\bf u}_n$ when $m=0$.
Using this notation, we obtain
\begin{Theorem}\label{ThI}
We have
\begin{equation}\label{S_I=rootzeta}
S_I=\zeta_{k-m+\ell+1}({\bf u}_1^I, {\bf u}_2^I, \cdots, {\bf u}_{k-m+\ell+1}^I, A_{k-m+\ell+1}),
\end{equation}
and therefore
\begin{equation}\label{sum1}
{\sum}_{\rm (i)}=\sum_{m=0}^{k}\sum_{I\in I_m^{[0, k]}}\zeta_{k-m+\ell+1}({\bf u}_1^I, {\bf u}_2^I, \cdots, {\bf u}_{k-m+\ell+1}^I, A_{k-m+\ell+1}).
\end{equation}
\end{Theorem}

\begin{proof}
It is enough to prove \eqref{S_I=rootzeta}.
In the sum $S_I$, all $a_i=0$ for $i\in I$.    Therefore
\begin{align*}
S_I &=\sum_{\substack{\alpha_i\in J_I \\(0\leq i\leq k-m)}}\sum_{\substack{a_{\alpha_i}\geq 1\\ (0\leq i\leq k-m)}}
\sum_{\substack{b_j\geq 1 \\(0\leq j\leq \ell-1)}}(b_{\ell-1}+\cdots+b_0)^{-s_{00}-s_{10}-\cdots-s_{\alpha_0-1,0}}
\\
&\quad \cdot  (a_{\alpha_0}+b_{\ell-1}+\cdots+b_0)^{-s_{\alpha_0,0}-s_{\alpha_0+1,0}-\cdots-
s_{\alpha_1-1,0}}\cdots\\
&\quad\cdot
(a_{\alpha_{k-m}}+\cdots+a_{\alpha_0}+b_{\ell-1}+\cdots+b_0)
^{-s_{\alpha_{k-m},0}-s_{\alpha_{k-m}+1,0}-\cdots-s_{k0}}\\
& \quad\cdot \prod_{j=1}^{\ell-1}(a_{\alpha_{k-m}}+\cdots+a_{\alpha_0}+b_{\ell-1}+\cdots +b_{j})^{-s_{k j}}
(a_{\alpha_{k-m}}+\cdots+a_{\alpha_0})^{-s_{k \ell}}\\
&=\sum_{\substack{\alpha_i\in J_I \\(0\leq i\leq k-m)}}\sum_{\substack{a_{\alpha_i}\geq 1\\ (0\leq i\leq k-m)}}
\sum_{\substack{b_j\geq 1 \\(0\leq j\leq \ell-1)}}
(b_{\ell-1}+\cdots+b_0)^{-z_0^I}
(a_{\alpha_0}+b_{\ell-1}+\cdots+b_0)^{-z_1^I}\cdots\\
&\quad\cdot(a_{\alpha_{k-m}}+\cdots+a_{\alpha_{0}}+b_{\ell-1}+\cdots+b_0)
^{-z_{k-m+1}^I}\\
& \quad\cdot \prod_{j=1}^{\ell-1}(a_{\alpha_{k-m}}+\cdots+a_{\alpha_0}+b_{\ell-1}+\cdots +b_{j})^{-s_{k j}}
(a_{\alpha_{k-m}}+\cdots+a_{\alpha_0})^{-s_{k \ell}}.
\end{align*}

Comparing this with the definition \eqref{A_def} of zeta-functions of root systems of
type $A_r$, we obtain the conclusion \eqref{S_I=rootzeta}.
\end{proof}
\begin{Remark}\label{continuation_of_bullet}
From Theorem \ref{ThII1} and Theorem \ref{ThI}, we see that 
\begin{align}\label{remarkzetabullet}
\zeta_{k+\ell+1, {k+1}}^{\bullet}({\bf u}_1, {\bf u}_2, \cdots, {\bf u}_{k+\ell+1}, A_{k+\ell+1})
\end{align}
can be written as a linear combination of zeta-functions of root systems.
Thanks to the meromorphic continuation of zeta-functions of root systems (\cite{M}), 
we conclude that \eqref{remarkzetabullet} can also be continued meromorphically to the whole space.
\end{Remark}
%
%
%
\subsection{Evaluation of ${\sum}_{\rm (ii)}$}

Next consider the case of $m_{k 0}>m_{00}+m_{k \ell}$.
We can write 
\begin{equation}\label{expressionk0}
m_{k 0}=m_{00}+m_{k \ell}+h \;(h\geq 1).
\end{equation}
So, for ${\pmb s}\in W_{\lambda}$
\begin{eqnarray}
{\sum}_{\rm (ii)}
&=&
\sum_{M\in {\rm{SSYT}}( \theta) {\rm and} {\rm (ii)}}
{m_{00}^{-s_{00}}m_{10}^{-s_{10}}\cdots m_{k-1,0}^{-s_{k-1,0}} m_{k 0}^{-s_{k 0}}m_{k 1}^{-s_{k 1}}\cdots m_{k \ell}^{-s_{k \ell}}}\notag\\
&=&
\sum_{\substack{m_{00},m_{k\ell}, h\geq 1\\
a_i\geq 0 (1\leq i\leq k-1)\\
b_j\geq 1(1\leq j\leq \ell-1)\\
1\leq m_{k-1, 0}\leq m_{k0}, m_{k0}>m_{k1}\geq 1}}m_{00}^{-s_{00}}(m_{00}+a_1)^{-s_{10}}\cdots (m_{00}+a_1+\cdots +a_{k-1})^{-s_{k-1,0}}  \notag\\
& &
(m_{00}+m_{k\ell}+h)^{-s_{k 0}}(m_{k \ell}+b_{\ell -1}+\cdots + b_1)^{-s_{k 1}}\cdots (m_{k \ell}+b_{\ell-1})^{-s_{k, \ell-1}} m_{k\ell}^{-s_{k\ell}},\notag\\
& & \label{original}\end{eqnarray}
where the sum on the right-hand side runs over $m_{00}, m_{k\ell},
a_i \;(1\leq i\leq k-1), b_j \;(1\leq j\leq \ell-1), h$
with the conditions $1\leq m_{k-1, 0}\leq m_{k0}, m_{k0}>m_{k1}\geq 1$, 
and it is to be noted that $m_{k-1, 0}=m_{00}+a_1+\cdots +a_{k-1}$, $m_{k0}=m_{00}+m_{k\ell}+h$ and $m_{k1}=m_{k \ell}+b_{\ell -1}+\cdots + b_1$.
\\
Define 
\begin{eqnarray*}
F_{\mu}&=&(m_{00}+a_1)^{-s_{10}}(m_{00}+a_1+a_2)^{-s_{20}}\cdots (m_{00}+a_1+\cdots +a_{\mu})^{-s_{\mu 0}},\\ 
P_{\mu}&=&(m_{00}+m_{k\ell}+h+p_1)^{-s_{k-1, 0}}\cdots (m_{00}+m_{k\ell}+h+p_1+\cdots+p_{\mu})^{-s_{k-\mu,0}},\\
Q_{\nu}&=&(m_{00}+m_{k\ell}+h+q_1)^{-s_{k1}}\cdots (m_{00}+m_{k\ell}+h+q_1+\cdots +q_{\nu})^{-s_{k \nu}}, \\
R_{\mu, \nu}&=& P_{\mu}(m_{00}+m_{k\ell}+h)^{-s_{k0}} Q_{\nu}\\
&=& (m_{00}+m_{k\ell}+h+p_1+\cdots+p_{\mu})^{-s_{k-\mu,0}}\cdots (m_{00}+m_{k\ell}+h+p_1)^{-s_{k-1, 0}}\\
&&\cdot (m_{00}+m_{k\ell}+h)^{-s_{k0}}(m_{00}+m_{k\ell}+h+q_1)^{-s_{k1}}\cdots (m_{00}+m_{k\ell}+h+q_1+\cdots +q_{\nu})^{-s_{k \nu}}, \\
G_{\nu}&=&(m_{k\ell}+b_{\ell-1}+b_{\ell-2}+\cdots+b_{\nu})^{-s_{k \nu}}(m_{k\ell}+b_{\ell-1}+b_{\ell-2}+\cdots+b_{\nu+1})^{-s_{k, \nu+1}} \\
  &&\qquad\cdots (m_{k\ell}+b_{\ell-1})^{-s_{k, \ell-1}}
\end{eqnarray*}
for $1\leq \mu\leq k-1$, where $p_i\geq 1$ and $1\leq \nu\leq \ell-1$, where $q_j\geq 0$. 
Then we have the following relation.
\begin{Lemma}\label{sumII}
For ${\pmb s}\in W_{\theta}^{\circ}$, we have
\begin{align*}
\lefteqn{{\sum}_{\rm (ii)}=\sum_{m_{00}, m_{k\ell}, h\geq 1}
\sum_{\mu=0}^{k-1}\sum_{\nu=0}^{\ell-1}
\sum_{\substack{a_i\geq 0(1\leq i\leq k-(\mu+1))\\p_i\geq 1(1\leq i\leq \mu)}}
\sum_{\substack{b_{j}\geq 1(\nu+1\leq j\leq \ell-1)\\
q_{j}\geq 0(1\leq j\leq \nu)}}(-1)^{\mu+\nu}}\\
&&\cdot
m_{00}^{-s_{00}}\cdot F_{k-(\mu+1)}\cdot R_{\mu, \nu}
\cdot G_{\nu+1} \cdot
m_{k\ell}^{-s_{k\ell}},
\end{align*}
where we set $F_0=G_{\ell}=1$ and  $P_0=Q_0=1$. 
\end{Lemma}
\begin{proof}
Using the notation $F_{\mu}$ and $G_{\nu}$, we see that \eqref{original} can be expressed as
\begin{equation}\label{FG1}
{\sum}_{\rm (ii)}=
\sum_{\substack{m_{00},m_{k\ell},h\geq 1\\
a_i\geq 0 (1\leq i\leq k-1)\\
b_j\geq 1 (1\leq j\leq \ell-1)\\
1\leq m_{k-1, 0}\leq m_{k0}, m_{k0}>m_{k1}\geq 1}}
m_{00}^{-s_{00}}\cdot F_{k-1}\cdot
(m_{00}+m_{k\ell}+h)^{-s_{k 0}}
\cdot G_1 \cdot
m_{k\ell}^{-s_{k\ell}}.
\end{equation}
Here we use the same type of decomposition as \eqref{basicdecomp} in the opposite direction.
The summation over $m_{k0}>m_{k1}$ may be 
decomposed as 
\begin{equation}\label{keystep2}
\sum_{m_{k0}>m_{k1}\geq 1}=\sum_{m_{k0}, m_{k1}\geq 1}-\sum_{1\leq m_{k0}\leq m_{k1}},
\end{equation} 
where in the first sum on the right-hand side no inequality between $m_{k0}$ and
$m_{k1}$ is imposed. Hence
for ${\pmb s}\in W_{\lambda}^{\circ}$,
the expression in
\eqref{FG1} becomes
\begin{eqnarray*}
{\sum}_{\rm (ii)}&=&
\sum_
{\substack{m_{00},m_{k\ell},h\geq 1\\
a_i\geq 0 (1\leq i\leq k-1)\\
b_j\geq 1 (1\leq j\leq \ell-1)\\
1\leq m_{k-1, 0}\leq m_{k0}
}}
m_{00}^{-s_{00}}\cdot F_{k-1}\cdot
(m_{00}+m_{k\ell}+h)^{-s_{k 0}}
\cdot G_1 \cdot
m_{k\ell}^{-s_{k\ell}}\\
&&- \sum_
{\substack{m_{00},m_{k\ell},h\geq 1\\
a_i\geq 0 (1\leq i\leq k-1)\\
b_j\geq 1 (1\leq j\leq \ell-1)\\
1\leq m_{k-1, 0}\leq m_{k0}\leq m_{k1}
}}
m_{00}^{-s_{00}}\cdot F_{k-1}\cdot
(m_{00}+m_{k\ell}+h)^{-s_{k 0}}
\cdot G_1 \cdot
m_{k\ell}^{-s_{k\ell}}.
\end{eqnarray*}
In the second sum, $m_{k0}\leq m_{k1}$ means $m_{k1}=m_{00}+m_{kl}+h+q_1$ ($q_1\geq 0$).
If we insert this expression into the term $m_{k1}$, then the inequality
$m_{k1}>m_{k2}$ (which was originally indicated by $b_1\geq 1$)
is no longer indicated, so we have to write this inequality explicitly.
Therefore we now obtain 
\begin{eqnarray*}
\lefteqn{{\sum}_{\rm (ii)}=
\sum_
{\substack{m_{00},m_{k\ell},h\geq 1\\
a_i\geq 0 (1\leq i\leq k-1)\\
1\leq m_{k-1, 0}\leq m_{k0}
}}
m_{00}^{-s_{00}}\cdot F_{k-1}\cdot
(m_{00}+m_{k\ell}+h)^{-s_{k 0}}}
\\
&& \cdot\left\{\sum_{b_j\geq 1(1\leq j\leq \ell-1)} G_1 \cdot
m_{k\ell}^{-s_{k\ell}}
-\sum_{\substack{b_j\geq 1(2\leq j\leq\ell-1), q_1\geq 0\\m_{k1}> m_{k2}\geq 1}}
(m_{00}+m_{k\ell}+h+q_1)^{-s_{k 1}}
\cdot G_2 \cdot
m_{k\ell}^{-s_{k\ell}}\right\}.
\end{eqnarray*}
A similar calculation shows
\begin{eqnarray*}
&&{\sum}_{\rm (ii)}=
\sum_
{\substack{m_{00},m_{k\ell},h\geq 1\\
a_i\geq 0 (1\leq i\leq k-1)\\
1\leq m_{k-1, 0}\leq m_{k0}
}}
m_{00}^{-s_{00}}\cdot F_{k-1}\cdot
(m_{00}+m_{k\ell}+h)^{-s_{k 0}}
\\
&& \cdot\left\{\sum_{\substack{b_j\geq 1\\(1\leq j\leq\ell-1)}} G_1 \cdot
m_{k\ell}^{-s_{k\ell}}
-
\left(\sum_{\substack{b_j\geq 1 (2\leq j\leq\ell-1)\\q_1\geq 0}}-\sum_{\substack{b_j\geq 1(2\leq j\leq\ell-1)\\
q_1\geq 0, 1\leq m_{k1}\leq m_{k2}}}\right)
(m_{00}+m_{k\ell}+h+q_1)^{-s_{k 1}}
\cdot G_2 \cdot
m_{k\ell}^{-s_{k\ell}}\right\}
\end{eqnarray*}
and 
$m_{k1}\leq m_{k2}$ means $m_{k2}=m_{00}+m_{k\ell}+h+q_1+q_2$ ($q_2\geq 0$), and so 
we have
\begin{eqnarray*}
\lefteqn{{\sum}_{\rm (ii)}=
\sum_
{\substack{m_{00},m_{k\ell},h\geq 1\\
a_i\geq 0 (1\leq i\leq k-1)\\
1\leq m_{k-1, 0}\leq m_{k0}
}}
m_{00}^{-s_{00}}\cdot F_{k-1}\cdot
(m_{00}+m_{k\ell}+h)^{-s_{k 0}}}
\\
&& \cdot\left\{\sum_{b_j\geq 1(1\leq j\leq\ell-1)} G_1 \cdot
m_{k\ell}^{-s_{k\ell}}
-\sum_{\substack{b_j\geq 1(2\leq j\leq\ell-1)\\
q_1\geq 0}}
Q_1\cdot G_2 \cdot
m_{k\ell}^{-s_{k\ell}}
+
\sum_{\substack{b_j\geq 1 (3\leq j\leq\ell-1)\\q_1,q_2\geq 0, m_{k2}> m_{k3}\geq 1}}
Q_2
\cdot G_3 \cdot
m_{k\ell}^{-s_{k\ell}}\right\}.
\end{eqnarray*} 
Apply the same argument repeatedly.     In the last stage we encounter the inequality
$m_{k,\ell-1}>m_{k\ell}$, where $m_{k,\ell-1}=m_{00}+m_{k\ell}+h+q_1+\cdots+q_{\ell-1}$,
but this inequality holds trivially.
Therefore we now obtain 
\begin{eqnarray*}
{\sum}_{\rm (ii)}&=&
\sum_
{\substack{m_{00},m_{k\ell},h\geq 1\\
a_i\geq 0 (1\leq i\leq k-1)\\
1\leq m_{k-1, 0}\leq m_{k0}
}}
m_{00}^{-s_{00}}\cdot F_{k-1}\cdot
(m_{00}+m_{k\ell}+h)^{-s_{k 0}}\\
&&\cdot\left\{
\sum_{\nu=0}^{ \ell-1}\sum_{\substack{b_{j}\geq 1 (\nu+1\leq j\leq \ell-1)\\
q_{j}\geq 0(1\leq j\leq \nu)}}
(-1)^{\nu}
Q_{\nu}\cdot G_{\nu+1} \cdot
m_{k\ell}^{-s_{k\ell}}\right\}.
\end{eqnarray*} 
Next, we take into consideration of the first summation, $\sum_{1\leq m_{k-1,0}\leq m_{k0}}$.
Again we apply the decomposition $\sum_{1\leq m_{k-1,0}\leq m_{k0}}=\sum_{m_{k0},m_{k1}\geq 1}-\sum_{m_{k-1,0}> m_{k0}\geq 1}$ to obtain 
\begin{eqnarray*}
{\sum}_{\rm (ii)}&=&
\left(\sum_{\substack{m_{00},m_{k\ell}, h\geq 1\\
a_i\geq 0 (1\leq i\leq k-1)}}-\sum_{\substack{m_{00},m_{k\ell}, h\geq 1\\
a_i\geq 0 (1\leq i \leq k-1)\\
m_{k-1,0}> m_{k0}\geq 1}}\right)
m_{00}^{-s_{00}}\cdot F_{k-1}\cdot
(m_{00}+m_{k\ell}+h)^{-s_{k 0}}\\
&&
\cdot\left\{\;\sum_{\nu=0}^{ \ell-1}\sum_
{\substack{b_{j}\geq 1 (\nu+1\leq j\leq \ell-1)\\
q_{j}\geq 0(1\leq j\leq \nu)}}
(-1)^{\nu}
Q_{\nu}\cdot G_{\nu+1} \cdot
m_{k\ell}^{-s_{k\ell}}\right\}.
\end{eqnarray*} 
As $m_{k-1,0}> m_{k0}\geq 1$ means 
$m_{k-1,0}=m_{00}+m_{k\ell}+h+p_1$ ($p_1\geq 1$) we have
\begin{eqnarray*}
{\sum}_{\rm (ii)}&=&
\left\{\sum_
{\substack{m_{00},m_{k\ell}, h\geq 1\\
a_i\geq 0 (1\leq i\leq k-1)}}
m_{00}^{-s_{00}}\cdot F_{k-1}\cdot
(m_{00}+m_{k\ell}+h)^{-s_{k 0}}\right.\\
&&
-\left.\sum_{\substack{m_{00},m_{k\ell}, h\geq 1\\a_i\geq 0 (1\leq i\leq k-2), p_1\geq 1\\1\leq m_{k-2,0}\leq m_{k-1,0}}}
m_{00}^{-s_{00}}\cdot F_{k-2}\cdot
(m_{00}+m_{k\ell}+h+p_1)^{-s_{k-1, 0}}
(m_{00}+m_{k\ell}+h)^{-s_{k 0}}\right\}
\\
&& 
\cdot\left\{\;\sum_{\nu=0}^{ \ell-1}\sum_
{\substack{b_{j}\geq 1 (\nu+1\leq j\leq \ell-1)\\
q_{j}\geq 0(1\leq j\leq \nu)}}(-1)^{\nu}
Q_{\nu}\cdot G_{\nu+1} \cdot
m_{k\ell}^{-s_{k\ell}}\right\}.
\end{eqnarray*} 
Repeating the similar calculation we obtain, for $m_{00}, m_{k\ell}, b_j, p_i, h\geq 1$ 
and $a_i, q_j \geq 0$,
\begin{align}
\lefteqn{{\sum}_{\rm (ii)}}\notag\\
&=
\sum_
{\substack{m_{00},m_{k\ell}, h\geq 1\\
a_i\geq 0 (1\leq i\leq k-1)}}
\left\{\;\sum_{\mu=0}^{k-1}\sum_{\substack{a_i\geq 0 (1\leq i \leq k-(\mu+1))\\
p_i\geq 1 (1\leq i \leq \mu)}}(-1)^{\mu}
m_{00}^{-s_{00}}\cdot F_{k-(\mu+1)}\cdot
P_{\mu}\cdot 
(m_{00}+m_{k\ell}+h)^{-s_{k 0}}\right\}\cdot \notag\\
&\qquad \left\{\; \sum_{\nu=0}^{ \ell-1}\sum_{\substack{b_{j}\geq 1 (\nu+1\leq j\leq \ell-1)\\
q_{j}\geq 0 (1\leq j\leq \nu)}}(-1)^{\nu}
Q_{\nu}\cdot G_{\nu+1} \cdot
m_{k\ell}^{-s_{k\ell}}\right\}\notag\\
&=
\sum_{m_{00}, m_{k\ell}, h\geq 1}
\left\{\;\sum_{\mu=0}^{k-1}\sum_{\nu=0}^{ \ell-1}\sum_{\substack{a_i\geq 0 (1\leq i\leq k-(\mu+1))\\p_i\geq 1(1\leq i\leq \mu)}}
\sum_{\substack{b_{j}\geq 1 (\nu+1\leq j\leq \ell-1)\\
q_{j}\geq 0 (1\leq j\leq \nu)}}
(-1)^{\mu+\nu}
m_{00}^{-s_{00}}\cdot F_{k-(\mu+1)}\cdot
R_{\mu, \nu}\cdot G_{\nu+1} \cdot
m_{k\ell}^{-s_{k\ell}}\right\}.
\label{tuika-tuika}
\end{align} 
This leads to the lemma.
\end{proof}

The above lemma gives the following expression for ${\sum}_{\rm (ii)}$.
\begin{Theorem}\label{ThII2}
Let ${\frak s}_1$,
${\frak v}$,
and
${\frak s}_2$ 
be as in the statement of Theorem \ref{HurwitzTheorem}.
Then for ${\pmb s}\in W_{\theta}^{\circ}$, we have
\begin{align*}
{\sum}_{\rm (ii)}=
\sum_{\mu=0}^{k-1}\sum_{\nu=0}^{\ell-1}(-1)^{\mu+\nu}Z_{\mu,\nu}({\pmb s};(k,\ell)),
\end{align*}
where $Z_{\mu,\nu}({\pmb s};(k,\ell))$ is defined by \eqref{Zmunu}.
\end{Theorem}
\begin{proof}\quad 
For any $\mu$, 
\begin{eqnarray*}
\sum_{a_i\geq 0 (1\leq i\leq k-( \mu+1))}
 F_{k-(\mu+1)}
&=&\sum_{a_i\geq 0 (1\leq i\leq k-( \mu+1))}
(m_{00}+a_1)^{-s_{10}}\cdots (m_{00}+a_1\cdots +a_{k-(\mu+1)})^{-s_{k-(\mu+1), 0}}\\
&=& \zeta^{\bullet, H}_{k-(\mu+1), k-(\mu+1)}({\frak s}_1, m_{00}, A_{k-(\mu+1)}),
\end{eqnarray*} 
where ${\frak s}_1=({\bf s}_{k-(\mu+1)}(s_{10}, 0), {\bf s}_{k-(\mu+1)-1}(s_{20}, 0), \cdots, {\bf s}_{2}(s_{k-(\mu+1)-1, 0}, 0), s_{k-(\mu+1), 0})$.\\
For any $\mu$ and $\nu$,
\begin{eqnarray*}
&&\sum_{p_i\geq 1 (1\leq i\leq \mu), q_j\geq 0 (1\leq j\leq \nu),h\geq 1}
R_{\mu, \nu}\\
&=&\sum_{p_i \geq 1(1\leq i\leq \mu), q_j\geq 0 (1\leq j\leq \nu),h\geq 1}
(m_{00}+m_{k\ell}+h+q_1+\cdots +q_{\nu})^{-s_{k\nu}}
\cdots 
(m_{00}+m_{k\ell}+h+q_1)^{-s_{k1}}\\
&& \hspace{1cm}
(m_{00}+m_{k\ell}+h)^{-s_{k 0}}
(m_{00}+m_{k\ell}+h+p_1)^{-s_{k-1, 0}}
\cdots \\
&& \hspace{1cm}
(m_{00}+m_{k\ell}+h+p_1+\cdots +p_{\mu})^{-s_{k-\mu, 0}}
\\
&=& \zeta_{\mu+\nu+1, \nu}^{\bullet, H}
({\frak v}, m_{00}+m_{k\ell}, A_{\mu+\nu+1}),
\end{eqnarray*} 
where ${\frak v}=({\bf v}_0, {\bf v}_1, \cdots, {\bf v}_{\mu+\nu+1})$.\\
Lastly, for any $\nu$,
\begin{eqnarray*}
\sum_{b_i \geq 0 (\nu+1\leq i\leq \ell-1)}
G_{\nu+1}
&=&\sum_{b_i\geq 0 (\nu+1\leq i\leq \ell-1)}
(m_{k\ell}+b_{\ell-1}+\cdots + b_{\nu+1})^{-s_{k\nu}}\cdots (m_{k\ell}+b_{\ell-1})^{-s_{k, \ell-1}}\\
&=&
\zeta_{\ell-(\nu+1)}^H({\frak s}_2, m_{k\ell}, A_{\ell-(\nu+1)}),
\end{eqnarray*} 
where ${\frak s}_2=({\bf s}_{\ell-\nu-1}(s_{k, {\ell-1}},0), {\bf s}_{\ell-\nu-2}(s_{k, {\ell-2}},0)\cdots, 
{\bf s}_{2}(s_{k, {\nu+2}},0),  s_{k,\nu+1})$.\\
Lemma \ref{sumII} with those calculations lead to the theorem.
\end{proof}

\subsection{An example}

In Example \ref{Ex4-1}, we described what Theorem \ref{HurwitzTheorem} implies
in the simplest case $(k,\ell)=(1,1)$.
To illustrate the contents of Theorem \ref{HurwitzTheorem} more, here we write down the
formulas for $\sum_{\rm (i)}$ and $\sum_{\rm (ii)}$ in the case of type 
$\theta={\rm rib}(2 | 3)$.
When
\hspace{1cm}${\pmb s}=
\ytableausetup{boxsize=normal}  
\begin{ytableau}
  \none & \none &   s_{2 3}\\
  \none & \none &   s_{2 2}\\
  \none & \none &   s_{21}\\
 s_{00} & s_{10}  & s_{2 0}
\end{ytableau}$\quad,
 \vspace{3mm}\\
\hspace{5mm}
$$\displaystyle{
 \zeta_{\theta}({\pmb s}) =
 \sum_{M\in {\rm{SSYT}}( \theta)}
 {m_{00}^{-s_{00}}m_{10}^{-s_{10}}m_{2 0}^{-s_{2 0}}m_{2 1}^{-s_{2 1}} m_{2 2}^{-s_{2 2}}m_{2 3}^{-s_{2 3}}}
 }.
 $$

Using Theorem \ref{ThI}, for $m_{20}\leq m_{00}+m_{23}$ we have
$$
\begin{array}{lll}
\displaystyle{{\sum}_{\rm (i)}}
&=&\displaystyle{\sum_{m=0}^2\sum_{I\in I_m^{[0, 2]}}\zeta_{6-m}({\bf u}_1^I, {\bf u}_2^I, \cdots, {\bf u}_{6-m}^I, A_{6-m})}\vspace{2mm}\\
&=&\displaystyle{\sum_{I\in I_0^{[0,2]}}\zeta_{6}({\bf u}_1^I, {\bf u}_2^I, \cdots, {\bf u}_{6}^I, A_{6})}\displaystyle{+\sum_{I\in I_1^{[0, 2]}}\zeta_{5}({\bf u}_1^I, {\bf u}_2^I, {\bf u}_3^I, {\bf u}_4^I, {\bf u}_5^I, A_{5})}\\
&&\qquad\qquad\displaystyle{+\sum_{I\in I_2^{[0, 2]}}\zeta_{4}({\bf u}_1^I, {\bf u}_2^I, {\bf u}_3^I, {\bf u}_4^I, A_{4})}\vspace{2mm}\\
&=&\displaystyle{\sum_{I\in I_0^{[0,2]}}\zeta_{6}({\bf s}_6(0, 0), {\bf s}_5(0, 0), {\bf s}_4(s_{23}, z_0^I), {\bf s}_3(s_{22}, z_{1}^I), {\bf s}_2(s_{21}, z_{2}^I), z_{3}^I, A_{6})}\vspace{2mm}\\
&&\displaystyle{+\sum_{I\in I_1^{[0, 2]}}\zeta_{5}({\bf s}_5(0, 0), {\bf s}_4(s_{23}, 0), {\bf s}_3(s_{22}, z_0^I), {\bf s}_2(s_{21}, z_{1}^I), z_{2}^I), A_{5})}\vspace{2mm}\\
&&\displaystyle{+\sum_{I\in I_2^{[0, 2]}}\zeta_{4}({\bf s}_4(s_{23}, 0), {\bf s}_3(s_{22}, 0), {\bf s}_2(s_{21}, z_0^I), z_{1}^I, A_{4})}
\end{array}
$$
(again we replace the commas between $s_i( \cdot)$ and $s_{i-1}(\cdot)$ by vertical bars).
$$
\begin{array}{lll}
\displaystyle{{\sum}_{\rm (i)}}
&=&\displaystyle{\zeta_{6}(0^{6} | 0^5 | s_{23}, 0, 0, 0 | s_{22}, 0, s_{00} | s_{21}, s_{10} | s_{20},A_6)}\vspace{2mm}\\
&&\displaystyle{+\zeta_{5}(0^5 | s_{23}, 0, 0, 0 | s_{22}, 0, s_{00} | s_{21}, s_{10} | s_{20}, A_5)}\vspace{2mm}\\
&&\displaystyle{+\zeta_{5}(0^5 | s_{23}, 0, 0, 0 | s_{22}, 0, 0 | s_{21}, s_{00}+s_{10} | s_{20}, A_5)}\vspace{2mm}\\
&&\displaystyle{+\zeta_{5}(0^5 | s_{23}, 0, 0, 0 | s_{22}, 0, 0 | s_{21}, s_{00} | s_{10}+s_{20}, A_5)}\vspace{2mm}\\
&&\displaystyle{
+\zeta_{4}(s_{23}, 0, 0, 0 | s_{22}, 0, 0 | s_{21}, s_{00}+s_{10} | s_{20}, A_4)}\vspace{2mm}\\
&&\displaystyle{
+\zeta_{4}(s_{23}, 0, 0, 0 | s_{22}, 0, 0 | s_{21}, s_{00} | s_{10}+s_{20}, A_4)}\vspace{2mm}\\
&&\displaystyle{
+\zeta_{4}(s_{23}, 0, 0, 0 | s_{22}, 0, 0 | s_{21}, 0 | s_{00}+s_{10}+s_{20}, A_4)}.
\end{array}
$$
And from Theorem \ref{ThII2}, we have 
\begin{eqnarray*}
{\sum}_{\rm (ii)}
&=& \sum_{m_{00}, m_{23}, h\geq 1}m_{00}^{-s_{00}}m_{23}^{-s_{23}}\\
&& 
\{\zeta_{1, 1}^{\bullet, H}(s_{10}, m_{00}, A_1)
\zeta_{1}^{H}(s_{20}, m_{00}+m_{k\ell}, A_1)
\zeta_{2}^{H}((s_{22}, 0 | s_{21}), m_{k\ell}, A_2)
\\
&&-
\zeta_{1, 1}^{\bullet, H}(s_{10}, m_{00}, A_1)
\zeta_{2, 1}^{\bullet, H}((0, s_{20} | s_{21}), m_{00}+m_{k\ell}, A_2)
\zeta_{1}^{H}(s_{22}, m_{k\ell}, A_1)
\\
&&+
\zeta_{1, 1}^{\bullet, H}(s_{10}, m_{00}, A_1)
\zeta_{3, 2}^{\bullet, H}((0, 0, s_{20} | 0, s_{21} | s_{22}) , m_{00}+m_{k\ell}, A_3)
\\
&&-
\zeta_{2}^{H}((s_{20}, 0 | s_{10}), m_{00}+m_{k\ell}, A_2)
\zeta_{2}^{H}((s_{22}, 0 | s_{21}), m_{k\ell}, A_2)
\\
&&+
\zeta_{3,1}^{\bullet, H}((0, s_{20}, 0 | s_{21}, s_{10} | 0), m_{00}+m_{k\ell}, A_3)
\zeta_{1}^{H}(s_{22}, m_{k\ell}, A_1)
\\
&&-
\zeta_{4,2}^{\bullet, H}((0, 0, s_{20}, 0 | 0, s_{21}, s_{10} | s_{22}, 0 | 0), m_{00}+m_{k\ell}, A_4)
\}.
\end{eqnarray*}
\section{The pictorial interpretation}\label{pictorial}

The following pictorial interpretation of the proofs developed in Sections
\ref{proof_FMT} and \ref{proof_SMT} may help the reader to improve the understandability.

In Section \ref{proof_FMT}, \eqref{basicdecomp} is the key step and we obtained  \eqref{pic5}. We may  
write it as
$$
\zeta_{\mathrm{rib}(k|\ell)}
\left(
\mbox{\raisebox{6mm}
{\ytableausetup{boxsize=1.7em}  
\begin{ytableau}
  \none & \none & \none&  *(gray) s_{k \ell}\\
  \none & \none & \none&  *(gray) \vdots\\
  \none & \none & \none & *(gray)  s_{k1}\\
 s_{00} & s_{10}  & \cdots & *(gray) s_{k 0}
\end{ytableau}}}\;
\right)
=
\zeta_{(k)}\left(
\mbox{\raisebox{-2mm}{\ytableausetup{boxsize=1.7em}  
\begin{ytableau}
 s_{00} & s_{10}  & \cdots & s_{k'0}
\end{ytableau}}}\;
\right)\cdot
\zeta_{(1^{\ell+1})}\left(
\mbox{\raisebox{6mm}{\ytableausetup{boxsize=1.7em}  
\begin{ytableau}
*(gray) s_{k \ell}\\
*(gray) \vdots\\
*(gray) s_{k1}\\
*(gray) s_{k 0}
\end{ytableau}}}\;
\right)-
\zeta_{\mathrm{rib}(k-1|\ell+1)}\left(
\mbox{\raisebox{9mm}{\ytableausetup{boxsize=1.7em}  
\begin{ytableau}
  \none & \none & \none&   *(gray) s_{k \ell}\\
  \none & \none & \none&   *(gray) \vdots\\
  \none & \none & \none &   *(gray) s_{k1}\\
  \none & \none & \none &   *(gray) s_{k0}\\
   s_{00} & s_{10}  & \cdots & s_{k'0}
\end{ytableau}}}\;
\right)
$$
where $k'=k-1$. This is actually a kind of harmonic product formulas for Schur
multiple zeta-functions, which is seen in Lemma 2.2 in \cite{BachY}.
Applying the same argument repeatedly to the Schur multiple 
zeta-function of anti-hook type on the right hand side we obtain the following relation, which is same as
\eqref{intermsofEZ}:
$$
\zeta_{\mathrm{rib}(k|\ell)}\left(
\mbox{\raisebox{6mm}{\ytableausetup{boxsize=1.7em} 
\begin{ytableau}
  \none & \none & \none&  *(gray) s_{k \ell}\\
  \none & \none & \none&  *(gray) \vdots\\
  \none & \none & \none & *(gray)  s_{k1}\\
 s_{00} & s_{10}  & \cdots & *(gray) s_{k 0}
\end{ytableau}}}\;
\right)=
 \sum_{i=0}^{k}(-1)^{k-i}
\zeta_{(i)}\left(
\mbox{\raisebox{-2mm}{\ytableausetup{boxsize=1.7em} 
\begin{ytableau}
 s_{00} & s_{10}  & \cdots & s_{i'0}
\end{ytableau}}}
\right)\cdot
\zeta_{(1^{\ell+k-i+1})}\left(
\mbox{\raisebox{17mm}{\ytableausetup{boxsize=1.7em} 
\begin{ytableau}
*(gray) s_{k \ell}\\
*(gray) \vdots\\
*(gray) s_{k1}\\
*(gray) s_{k 0}\\
 s_{k'0}\\
 \vdots\\
 s_{i''0}\\
 s_{i0}
\end{ytableau}}}\;
\right),
$$
where $i'=i-1, i''=i+1$.

To obtain the pictorial interpretation for the calculations in Section \ref{proof_SMT}, we introduce the following new kinds zeta-functions of Hurwitz type.
For  ${\pmb s}=(s_{ij})\in T(\lambda,\mathbb{C})$ and $x>0$, the Schur multiple zeta-function of Hurwitz type associated with $\lambda$ is defined by the series 
$$
\zeta^H_{\lambda}({ \pmb s}, x)=\sum_{M\in \mathrm{SSYT}(\lambda)}
{(M+x)^{ -\pmb s}}, 
$$
where $(M+x)^{ -\pmb s}=\displaystyle{\prod_{(i, j)\in \lambda}(m_{ij}+x)^{-s_{ij}}}$ for $M=(m_{ij})\in \mathrm{SSYT}(\lambda)$. 
Here, each $m_{ij}\in\mathbb{N}$ by definition.
If we allow some $m_{ij}=0$, we call $M=(m_{ij})$ a modified semi-standard Young tableau,
and for such a tableau we may define the zeta-function of Hurwitz type in the same way
as above.
Let $(I, J)$ be the most top left box of $M$.   If $m_{IJ}\geq 0$ (resp. $m_{IJ}=0$) instead of $m_{IJ}\geq 1$, then we write the asssociated zeta-function as 
$
\zeta_{\lambda}^{H, 0}({ \pmb s}, x)
$ (resp.$
\zeta_{\lambda}^{H, 00}({ \pmb s}, x)
$) instead.
Furthermore,
for $p>0$, the finite Hurwiz-type Schur multiple zeta-function associated with $\lambda$ is defined by the series 
$$
\zeta^H_{\lambda}({ \pmb s}, x)_{\leq p}=\sum_{M\in \mathrm{SSYT}(\lambda)_{\leq p-x}}
{(M+x)^{ -\pmb s}}, 
\quad
\zeta^H_{\lambda}({ \pmb s}, x)_{< p}=\sum_{M\in \mathrm{SSYT}(\lambda)_{< p-x}}
{(M+x)^{ -\pmb s}}, 
$$
where $(M+x)^{ -\pmb s}=\displaystyle{\prod_{(i, j)\in \lambda}(m_{ij}+x)^{-s_{ij}}}$ for $M=(m_{ij})\in \mathrm{SSYT}(\lambda)_{\leq p-x}$
or for $M=(m_{ij})\in \mathrm{SSYT}(\lambda)_{< p-x}$.
Here, $\mathrm{SSYT}(\lambda)_{\leq p-x}$ (resp. $\mathrm{SSYT}(\lambda)_{< p-x}$) means that any element $(m_{ij})$ satisfies $m_{ij}+x\leq p$
(resp. $m_{ij}+x< p$).
In case of modified semi-standard Young tableaux, we write the associated zeta-function by
$
\zeta_{\lambda}^{H, 0}({ \pmb s}, x)_{\leq p}
$
( resp. $
\zeta_{\lambda}^{H, 0}({ \pmb s}, x)_{< p}
$)
or 
$
\zeta_{\lambda}^{H, 00}({ \pmb s}, x)_{\leq p}
$
(resp. $
\zeta_{\lambda}^{H, 00}({ \pmb s}, x)_{< p}
$).

Then, using the above notation, we find that \eqref{FG1} can be expressed as 
$$
{\sum}_{\rm (ii)}=
\sum_{m_{00},m_{k\ell},h\geq 1}
\zeta_{(k)}^{H, 00}\left(
\mbox{\raisebox{-2mm}{\ytableausetup{boxsize=1.7em} 
\begin{ytableau}
 s_{00} & s_{10}  & \cdots & s_{k' 0}
\end{ytableau}}}\;
, m_{00}\right)_{\leq m_{00}+m_{k\ell}+h}\cdot
(m_{00}+m_{k\ell}+h)^{-s_{k 0}}
$$
$$\cdot 
\zeta^{H, 00}_{(1^{\ell})}\left(
\mbox{\raisebox{6mm}{\ytableausetup{boxsize=1.7em} 
\begin{ytableau}
s_{k\ell}\\
\vdots\\
s_{k2}\\
s_{k1}\\
\end{ytableau}}}\;, m_{k\ell}\right)_{<m_{00}+m_{k\ell}+h},
$$
where $k'=k-1$.

In the proof of Lemma \ref{sumII}, the key decomposition is \eqref{keystep2}.
Using this decomposition repeatedly, we arrive at \eqref{tuika-tuika}, which can be
written as
\begin{align}
{\sum}_{\rm (ii)}&=
\sum_{m_{00},m_{k\ell},h\geq 1}
\sum_{\mu=0}^{k-1}(-1)^{\mu}
 \zeta_{(k-\mu)}^{H, 00}\left(
\mbox{\raisebox{-2mm}{\ytableausetup{boxsize=1.8em} 
\begin{ytableau}
 s_{00} & s_{10}  & \cdots & s_{k_\mu' 0}
\end{ytableau}}}\;
, m_{00}\right)\notag\\
& \cdot
\zeta^H_{(1^{\mu})}\left(
\mbox{\raisebox{6mm}{\ytableausetup{boxsize=1.8em} 
\begin{ytableau}
s_{k'0}\\
\vdots\\
s_{k_\mu'' 0}\\
s_{k_\mu 0}\\
\end{ytableau}}}\;, m_{00}+m_{k\ell}+h\right)
\cdot (m_{00}+m_{k\ell}+h)^{-s_{k 0}}\notag\\
& \cdot 
\sum_{\nu=0}^{\ell-1}(-1)^{\nu}
\zeta^{H,0}_{(\nu)}\left(
\mbox{\raisebox{-2mm}{\ytableausetup{boxsize=1.8em} 
\begin{ytableau}
s_{k1} & s_{k2} & \cdots & s_{k\nu}
\end{ytableau}}}\;, m_{00}+m_{k\ell}+h\right) \cdot 
\zeta^{H, 00}_{(1^{\ell-\nu})}\left(
\mbox{\raisebox{6mm}{\ytableausetup{boxsize=1.8em} 
\begin{ytableau}
s_{k\ell}\\
\vdots\\
s_{k\nu''}\\
s_{k\nu'}\\
\end{ytableau}}}\;, m_{k\ell}\right)\notag
 \\
&=\sum_{m_{00}, m_{k\ell}, h\geq 1}
\sum_{\mu=0}^{k-1}\sum_{\nu=0}^{\ell-1}
(-1)^{\mu+\nu}
\cdot
\zeta_{(k-\mu)}^{H,00}\left(
\mbox{\raisebox{-2mm}{\ytableausetup{boxsize=1.8em} 
\begin{ytableau}
 s_{00} & s_{10}  & \cdots & s_{k_{\mu}' 0}
\end{ytableau}}}\;
, m_{00}\right)\notag\\
&\cdot 
\zeta^{H, 00}_{(\nu+1, 1^{\mu})}\left(
\mbox{\raisebox{6mm}{\ytableausetup{boxsize=1.8em} 
\begin{ytableau}
s_{k0} & s_{k1} & \cdots & s_{k\nu}\\
s_{k'0}\\
\vdots\\
s_{k_{\mu}0}\\
\end{ytableau}}}\;
, m_{00}+m_{k\ell}+h\right)%
%
\cdot 
\zeta^{H, 00}_{(1^{\ell-\nu})}\left(
\mbox{\raisebox{6mm}{\ytableausetup{boxsize=1.8em} 
\begin{ytableau}
s_{k\ell}\\
\vdots\\
s_{k\nu''}\\
s_{k\nu'}\\
\end{ytableau}}}\;, m_{k\ell}\right),\label{picZ}
\end{align}
where $k'=k-1$, $k_\mu=k-\mu$, $k_\mu'=k-(\mu+1)$, $k_{\mu}''=k-\mu+1$, $\nu'=\nu+1$ and $\nu''=\nu+2$.
This is the expression in Lemma \ref{sumII}.

Moreover, the last expression \eqref{picZ} shows
\begin{align}
Z_{ \mu, \nu}=& \sum_{m_{00}, m_{k\ell}, h\geq 1}
\zeta_{(k-\mu)}^{H,00}\left(
\mbox{\raisebox{-2mm}{\ytableausetup{boxsize=1.8em} 
\begin{ytableau}
 s_{00} & s_{10}  & \cdots & s_{k_{\mu}' 0}
\end{ytableau}}}\;
, m_{00}\right)\cdot\notag\\
& 
\zeta^{H, 00}_{(\nu+1, 1^{\mu})}\left(
\mbox{\raisebox{6mm}{\ytableausetup{boxsize=1.8em} 
\begin{ytableau}
s_{k0} & s_{k1} & \cdots & s_{k\nu}\\
s_{k'0}\\
\vdots\\
s_{k_{\mu}0}\\
\end{ytableau}}}\;
, m_{00}+m_{k\ell}+h\right)%
%
\cdot 
\zeta^{H, 00}_{(1^{\ell-\nu})}\left(
\mbox{\raisebox{6mm}{\ytableausetup{boxsize=1.8em} 
\begin{ytableau}
s_{k\ell}\\
\vdots\\
s_{k\nu''}\\
s_{k\nu'}\\
\end{ytableau}}}\;, m_{k\ell}\right),\label{picZZ}
\end{align}
which gives a pictiroal interpretation of Theorem \ref{ThII2}.

From \eqref{picZ} and \eqref{picZZ} we see that Lemma \ref{sumII} and Theorem \ref{ThII2},
which represent the most essential part of the proof of Theorem \ref{HurwitzTheorem},
may be regarded as a procedure of expressing a Schur multiple zeta-function of
anti-hook type in terms of zeta-functions of hook type.

\section{The cases $k+\ell\leq 3$}
When $k+\ell\leq 3$, the structure of $\zeta_{\theta}({\pmb s})$ is rather simple,
and in these cases, we can develop an alternative simpler way of computations.
In this final section we present such computations.    
The alternative argument is embodied in the process (ii) in the following examples.
Note that this method
cannot be applied to the case $k+\ell>3$. 

\begin{Example}
We consider the case $(k, \ell)=(1, 1)$.\\
When
\hspace{1cm}${\pmb s}=
\ytableausetup{boxsize=normal}  
\begin{ytableau}
  \none &   s_{11}\\
 s_{00} & s_{10}
\end{ytableau}$\quad, 
 \vspace{3mm}\\
\begin{equation}\label{rootzeta11}
 \zeta_{\theta}({\pmb s}) =
 \sum_{M\in {\rm{SSYT}}( \theta)}
 {m_{00}^{-s_{00}}m_{10}^{-s_{10}}m_{11}^{-s_{11}}}.
\end{equation}
(i) Consider the case $m_{10}\leq m_{00}+m_{11}$.
From Theorem \ref{ThII1} and Theorem \ref{ThI}, we have
\begin{align*}
{\sum}_{\rm (i)}
&=\zeta_{3, 2}^{\bullet}(0^3 | s_{11}, s_{00} | s_{10}, A_3)\\
&=\zeta_3(0^3 | s_{11}, s_{00} | s_{10}, A_3)+\zeta_2(s_{11}, 0 | s_{00}+s_{10}, A_2)
+\zeta_2(s_{11}, s_{00} | s_{10}, A_2).
\end{align*}
%
(ii) Consider the case $m_{10}>m_{00}+m_{11}$.
When we write $m_{10}=m_{00}+m_{11}+h$ ($h\geq 1$),
\begin{eqnarray*}
{\sum}_{\rm (ii)}&=&\sum_{\substack{M\in {\rm{SSYT}}( \theta)\\m_{10}> m_{00}+m_{11}}}
{m_{00}^{-s_{00}}m_{10}^{-s_{10}}m_{11}^{-s_{11}}}
=
\sum_{\substack{m_{00}, m_{11}, h\geq 1}}
{m_{00}^{-s_{00}}(m_{00}+m_{11}+h)^{-s_{10}}m_{11}^{-s_{11}}}\\
&=&\zeta_3(s_{00}, s_{11}, 0 | 0^2 | s_{10}, A_3).
\end{eqnarray*}
%
\noindent Combining (i) and (ii), we have
\begin{align}\label{ex0}
&\zeta_{\theta}({\pmb s})
=\zeta_{3, 2}^{\bullet}(0^3 | s_{11}, s_{00} | s_{10}, A_3)
+\zeta_3(s_{00}, s_{11}, 0 | 0^2 | s_{10}, A_3)\\
&=\zeta_3(0^3 | s_{11}, s_{00} | s_{10}, A_3)+\zeta_2(s_{11}, 0 | s_{00}+s_{10}, A_2)
+\zeta_2(s_{11}, s_{00} | s_{10}, A_2)
+\zeta_3(s_{00}, s_{11}, 0 | 0^2 | s_{10}, A_3).\notag
\end{align}
This agrees with \eqref{example4-1}.

\end{Example}
\noindent
\begin{Example} 
We consider the case $(k, \ell)=(1, 2)$.\\
When
\hspace{1cm}${\pmb s}=
\ytableausetup{boxsize=normal}  
\begin{ytableau}
  \none &   s_{12}\\
  \none &   s_{11}\\
 s_{00} & s_{10}
\end{ytableau}$\quad,
 \vspace{3mm}\\
\begin{equation}\label{rootzeta11}
 \zeta_{\theta}({\pmb s}) =
 \sum_{M\in {\rm{SSYT}}( \theta)}
 {m_{00}^{-s_{00}}m_{10}^{-s_{10}}m_{11}^{-s_{11}}m_{12}^{-s_{12}}}.
\end{equation}
(i) Consider the case $m_{10}\leq m_{00}+m_{12}$. 
From Theorem \ref{ThII1}, 
\begin{eqnarray*}
{\sum}_{\rm (i)}
&=&\zeta_{4, 2}^{\bullet}(0^4 | s_{12}, 0^2| s_{11}, s_{00} | s_{10}, A_4).
\end{eqnarray*}
%
(ii) Consider the case $m_{10}>m_{00}+m_{12}$.
When we write $m_{10}=m_{00}+m_{12}+h$ ($h\geq 1$),
\begin{eqnarray*}
{\sum}_{\rm (ii)}
=
\sum_{\substack{m_{00}, m_{12}, b_1, h\geq 1\\1\leq b_1<m_{00}+h}}
{m_{00}^{-s_{00}}(m_{00}+m_{12}+h)^{-s_{10}}(m_{12}+b_1)^{-s_{11}}m_{12}^{-s_{12}}}.
\end{eqnarray*}
When  $h\geq b_1$, we put $h=b_1+q_1$ ($q_1\geq 0$) and so, $m_{10}=m_{00}+m_{12}+b_1+q_1$.
Then
\begin{eqnarray*}
{\sum_{h\geq b_1}}_{\substack{\vspace{-6mm}\\\rm (ii)}}&=&
\sum_{m_{00}, m_{12}, b_1\geq 1, q_1\geq 0}
{m_{00}^{-s_{00}}(q_1+m_{00}+m_{12}+b_1)^{-s_{10}}(m_{12}+b_1)^{-s_{11}}m_{12}^{-s_{12}}}\\
&=&
\zeta_{4,1}^{\bullet}(0, s_{00}, s_{12}, 0 | 0, 0, s_{11} | 0^2 | s_{10}, A_4).
\end{eqnarray*}
When  $0<h< b_1$, we put $b_1=h+p_1$ ($p_1\geq 1$).
Then the condition $b_1<m_{00}+h$ implies $p_1<m_{00}$, hence we put $m_{00}=p_1+b_0$ 
($b_0\geq 1$) and so,
\begin{eqnarray*}
{\sum_{h< b_1}}_{\substack{\vspace{-6mm}\\\rm (ii)}}
&=&
\sum_{m_{12}, h, p_1, b_0\geq 1}
{(p_1+b_0)^{-s_{00}}(p_1+b_0+m_{12}+h)^{-s_{10}}(m_{12}+h+p_1)^{-s_{11}}m_{12}^{-s_{12}}}\\
&=&
\sum_{m_{12}, h, p_1, b_0\geq 1}
{(b_0+p_1)^{-s_{00}}(b_0+p_1+m_{12}+h)^{-s_{10}}(p_1+m_{12}+h)^{-s_{11}}m_{12}^{-s_{12}}}\\
&=&
\zeta_4(0, 0, s_{12}, 0 | s_{00}, 0, 0 | 0, s_{11} | s_{10}, A_4).
\end{eqnarray*}
%
%
\noindent Combining (i) and (ii), we have
\begin{eqnarray}\label{ex1}
\zeta_{\theta}({\pmb s})
&=&\zeta_{4, 2}^{\bullet}(0^4 | s_{12}, 0^2| s_{11}, s_{00} | s_{10}, A_4)
+\zeta_{4,1}^{ \bullet}(0, s_{00}, s_{12}, 0 | 0, 0, s_{11} | 0^2 | s_{10}, A_4)\\
&+&\zeta_4(0, 0, s_{12}, 0 | s_{00}, 0, 0 | 0, s_{11} | s_{10}, A_4).\notag
\end{eqnarray}

We now compare this result with Theorem \ref{HurwitzTheorem}.
The case $(k,\ell)=(1,2)$ of Theorem \ref{HurwitzTheorem} implies that
$$
\zeta_{\theta}({\pmb s})
=\zeta_{4, 2}^{\bullet}(0^4 | s_{12}, 0^2| s_{11}, s_{00} | s_{10}, A_4)
+Z_{00}({\pmb s};(1,2))-Z_{01}({\pmb s};(1,2)).
$$
From this and \eqref{ex1}, we obtain the following identity:
\begin{align}\label{identity-1}
Z_{00}({\pmb s};(1,2))-&Z_{01}({\pmb s};(1,2))=
\zeta_{4,1}^{ \bullet}(0, s_{00}, s_{12}, 0 | 0, 0, s_{11} | 0^2 | s_{10}, A_4)\\
&+\zeta_4(0, 0, s_{12}, 0 | s_{00}, 0, 0 | 0, s_{11} | s_{10}, A_4).\notag
\end{align}

It is possible to verify this fact directly.
In fact, by definition,
$$
Z_{00}({\pmb s};(1,2))=\sum_{m_{00},m_{12}\geq 1}m_{00}^{-s_{00}}m_{12}^{-s_{12}}
\sum_{h\geq 1}(m_{00}+m_{12}+h)^{-s_{10}}\sum_{b_1\geq 1}(m_{12}+b_1)^{-s_{11}}.
$$
Divide this sum into two parts, according to the conditions
$b_1<m_{00}+h$ and $b_1\geq m_{00}+h$.
The former part is exactly ${\sum}_{\rm (ii)}$ and hence is equal to the right-hand side
of \eqref{identity-1} by the above argument.
The latter part is equal to $Z_{01}({\pmb s};(1,2))$ by its definition.
Therefore \eqref{identity-1} follows.
\end{Example}

\noindent
\begin{Example} 
We consider the case $(k, \ell)=(2, 1)$.\\
When
\hspace{1cm}${\pmb s}=
\ytableausetup{boxsize=normal}  
\begin{ytableau}
  \none &   \none &   s_{21}\\
 s_{00} & s_{10} & s_{20}
\end{ytableau}$\quad,
 \vspace{3mm}\\
\begin{equation}\label{rootzeta11}
 \zeta_{\theta}({\pmb s}) =
 \sum_{M\in {\rm{SSYT}}( \theta)}
 {m_{00}^{-s_{00}}m_{10}^{-s_{10}}m_{20}^{-s_{20}}m_{21}^{-s_{21}}}.
\end{equation}
(i) Consider the case $m_{20}\leq m_{00}+m_{21}$.
From Theorem \ref{ThII1},
\begin{eqnarray*}
{\sum}_{\rm (i)}
&=&\zeta_{4, 3}^{\bullet}(0^4 | 0^2, s_{00}| s_{21}, s_{10} | s_{20}, A_4).
\end{eqnarray*}
%
(ii) Consider the case $m_{20}>m_{00}+m_{21}$.
When we write $m_{20}=m_{00}+m_{21}+h$ ($h\geq 1$),
\begin{eqnarray*}
{\sum}_{\rm (ii)}
=
\sum_{\substack{m_{00}, m_{21}, h\geq 1\\0\leq a_1\leq m_{21}+h}}
{m_{00}^{-s_{00}}(m_{00}+a_1)^{-s_{10}}(m_{00}+m_{21}+h)^{-s_{20}}m_{21}^{-s_{21}}}.
\end{eqnarray*}
%
When  $h\geq a_1$, we put $h=a_1+q_1$ ($q_1\geq 0$) and so, $m_{20}=m_{00}+m_{21}+a_1+q_1$.
Then
\begin{eqnarray*}
{\sum_{h\geq a_1}}_{\substack{\vspace{-6mm}\\\rm (ii)}}
&=&
\sum_{\substack{m_{00}, m_{21}\geq 1\\ 
a_1, q_1\geq 0}}
{m_{00}^{-s_{00}}(m_{00}+a_1)^{-s_{10}}(m_{00}+m_{21}+a_1+q_1)^{-s_{20}}m_{21}^{-s_{21}}}\\
&=&
\sum_{\substack{m_{00}, m_{21}\geq 1\\
a_1, q_1\geq 0}}
{m_{00}^{-s_{00}}(a_1+m_{00})^{-s_{10}}(q_1+a_1+m_{00}+m_{21})^{-s_{20}}m_{21}^{-s_{21}}}\\
&=&
\zeta_{4, 2}^{\bullet}(0, 0, s_{00}, s_{21} | 0, s_{10}, 0 | 0^2 | s_{20}, A_4).
\end{eqnarray*}
When  $1\leq h< a_1$, we put $a_1=h+p_1$ ($p_1\geq 1$).
Then the condition $a_1\leq m_{21}+h$ implies $p_1\leq m_{21}$, hence we put
$m_{21}=p_1+a_2$ ($a_2\geq 0$) and so,
\begin{eqnarray*}
{\sum_{h< a_1}}_{\substack{\vspace{-6mm}\\\rm (ii)}}
&=&
\sum_{\substack{m_{00}, h, p_1\geq 1\\
a_2\geq 0}}
{m_{00}^{-s_{00}}(m_{00}+h+p_1)^{-s_{10}}(m_{00}+p_1+a_2+h)^{-s_{20}}(p_1+a_2)^{-s_{21}}}\\
&=&
\sum_{\substack{m_{00}, h, p_1\geq 1\\
a_2\geq 0}}
{m_{00}^{-s_{00}}(p_1+m_{00}+h)^{-s_{10}}(a_2+p_1+m_{00}+h)^{-s_{20}}(a_2+p_1)^{-s_{21}}}\\
&=&
\zeta_{4, 1}^{\bullet}(0, 0, s_{00}, 0 | s_{21}, 0, 0 | 0, s_{10} | s_{20}, A_4).
\end{eqnarray*}

Combining (i) and (ii), we have
\begin{eqnarray*}\label{ex1bis}
\zeta_{\theta}({\pmb s})
&=&\zeta_{4, 3}^{\bullet}(0^4 | 0^2, s_{00}| s_{21}, s_{10} | s_{20}, A_4)\\
&+&\zeta_{4, 2}^{ \bullet}(0, 0, s_{00}, s_{21} | 0, s_{10}, 0 | 0^2 | s_{20}, A_4)
+\zeta_{4, 1}^{ \bullet}(0, 0, s_{00}, 0 | s_{21}, 0, 0 | 0, s_{10} | s_{20}, A_4).
\end{eqnarray*}

Comparing this result with Theorem \ref{HurwitzTheorem}, we obtain
\begin{align}\label{identity-2}
Z_{00}({\pmb s};(2,1))-&Z_{10}({\pmb s};(2,1))=
\zeta_{4,2}^{ \bullet}(0, 0, s_{00}, s_{21} | 0, s_{10}, 0 | 0^2 | s_{20}, A_4)\\
&+\zeta_4(0, 0, s_{00}, 0 | s_{21}, 0, 0 | 0, s_{10} | s_{20}, A_4).\notag
\end{align}
We can verify this fact directly, as in the case of \eqref{identity-1}.

\end{Example}
\bigskip

{\bf Acknowledgments}
The authors express their thanks to Professor Hideki Murahara, Professor Shuji Yamamoto
and the anonymous referee
for useful comments.
%
%
%
%

%
\bigskip
\noindent
\textsc{Kohji Matsumoto}\\
Graduate School of Mathematics, \\
Nagoya University, \\
Furo-cho, Chikusa-ku, Nagoya, 464-8602, Japan \\
 \texttt{kohjimat@math.nagoya-u.ac.jp}

\medskip

\noindent
\textsc{Maki Nakasuji}\\
Department of Information and Communication Science, Faculty of Science, \\
 Sophia University, \\
 7-1 Kio-cho, Chiyoda-ku, Tokyo, 102-8554, Japan \\
 \texttt{nakasuji@sophia.ac.jp}

\end{document}